\documentclass{article}

\usepackage{amsmath,amsfonts,amsthm,amssymb,mathtools}
\usepackage{biblatex}
\usepackage{bbm}
\usepackage{enumitem}
\usepackage{graphicx,subcaption}
\usepackage{aliascnt}
\usepackage{hyperref}
\usepackage{cleveref}
\usepackage{xcolor}

\newtheorem{theorem}{Theorem}[section]

\newaliascnt{lemma}{theorem}
\newtheorem{lemma}[lemma]{Lemma}
\aliascntresetthe{lemma}

\newaliascnt{proposition}{theorem}
\newtheorem{proposition}[proposition]{Proposition}
\aliascntresetthe{proposition}

\newaliascnt{corollary}{theorem}

\aliascntresetthe{corollary}

\theoremstyle{definition}
\newaliascnt{definition}{theorem}

\aliascntresetthe{definition}

\theoremstyle{remark}
\newtheorem*{remark}{Remark}
\numberwithin{equation}{section}
\numberwithin{theorem}{section}

\crefname{theorem}{theorem}{theorems}
\Crefname{theorem}{Theorem}{Theorems}
\crefname{lemma}{lemma}{lemmas}
\Crefname{lemma}{Lemma}{Lemmas}
\crefname{proposition}{proposition}{propositions}
\Crefname{proposition}{Proposition}{Propositions}
\crefname{corollary}{corollary}{corollaries}
\Crefname{corollary}{Corollary}{Corollaries}
\crefname{definition}{definition}{definitions}
\Crefname{definition}{Definition}{Definitions}

\allowdisplaybreaks

\title{Phase Transition of a Semiflexible Membrane in Two Dimensions}
\author{Jingu Hwang\thanks{\texttt{j9hwang@kaist.ac.kr}}}
\date{}
\begin{document}
\maketitle
\begin{abstract}
    We study a two-dimensional semiflexible membrane model whose formal Hamiltonian is given by $H[\phi]=\sum_x \bigl(\|\nabla \phi_x\|^2 +N^\lambda |\Delta \phi_x|^2\bigr)$, interpolating between the discrete Gaussian free field (DGFF) and the membrane model (MM). We analyze its finite-volume covariance in the bulk as $N\to\infty$, and identify distinct regimes depending on the parameter $\lambda$.
    
    For $\lambda<0$, the covariance of the model agrees with that of the DGFF up to a negligible error, while for $\lambda>2$, it agrees with the rescaled MM covariance up to a negligible error. In the intermediate regime $\lambda\in[0,2]$, we identify a different crossover behavior: in the microscopic range $\|x-y\|\ll N^{\lambda/2}$, the leading asymptotics no longer resolve the precise distance between the two points. In this microscopic regime, we further determine the leading logarithmic coefficient of the bulk covariance. These results provide a unified description of the crossover from gradient-dominated to curvature-dominated behavior in this class of models.
\end{abstract}
    % !TEX root = ../main.tex
\section{Introduction}
\subsection{Model and Main Result}
For each positive integer $N$, consider the finite box $V_N\coloneq \{-N,\dots,N\}^d$ in $\mathbb{Z}^d$. The discrete Laplacian $\Delta$ is defined by
\[\Delta f(x)=\frac{1}{2d}\sum_{y:\|x-y\|_1=1}f(y)-f(x),\]
and $\Delta^2$ is defined as $\Delta^2=\Delta\circ \Delta$. Let $-\Delta_N$ be the operator acting on functions that vanish outside $V_N$, and define $\Delta^2_N$ similarly. We emphasize that $\Delta^2_N$ is not the same as $(\Delta_N)^2$: the former is the restriction of the whole-space operator $\Delta^2$, while the latter applies the restricted operator twice. The model of this paper is the Gaussian interface $\{\varphi_x\}_{x\in V_N}$ whose covariance is the Green function of an operator interpolating between the discrete Laplacian $-\Delta_N$ and the discrete bi-Laplacian $\Delta^2_N$. 

Before introducing the interpolation model, we briefly review the two constituent models associated with these two operators: the discrete Gaussian free field (DGFF) and the membrane model (MM). DGFF is the centered Gaussian field defined via the (formal) Hamiltonian $H^\nabla[\phi]=\frac{1}{2}\sum|\nabla\phi_x|^2$, and this Hamiltonian describes the energy associated with height fluctuations. Its covariance can be written as $(-\Delta_N)^{-1}$. It exhibits logarithmic correlations when $d=2$, and is localized when $d\ge3$. Its fluctuation behavior and extremal properties have been extensively studied. See, e.g., \autocite{biskup2016extreme, bolthausen2001entropic, bramson2016convergence, daviaud2006extremes}.

MM is the centered Gaussian field whose covariance is the Green function of the discrete bi-Laplacian, and its formal Hamiltonian is given by $H^\Delta[\phi]=\frac{1}{2}\sum|\Delta\phi_x|^2$. Its covariance is $(\Delta^2_N)^{-1}$, and MM exhibits logarithmic correlation when $d=4$ and is localized when $d\ge5$. It is less well understood, partly because fewer analytic and probabilistic tools are available. For recent studies of MM, see, e.g., \autocite{kurt2009maximum, schweiger2020maximum}.

A fundamental difference between DGFF and MM is that DGFF is log-correlated when $d=2$, while such behavior for MM occurs when $d=4$. It is therefore natural to investigate the behavior of the combination of these two models. This leads to the random interface model whose formal Hamiltonian is given by
\[H[\phi]=\frac{1}{2}\sum_x\bigl[\kappa_1|\nabla \phi_x|^2+\kappa_2|\Delta\phi_x|^2\bigr],\]
for $\kappa_1,\kappa_2>0$, i.e., the mixed model of DGFF and MM. Here $\kappa_1$ and $\kappa_2$ represent the \emph{lateral tension} and \emph{bending rigidity}, respectively. See, e.g., \autocite{cipriani2020scaling} and the references therein. We may let $\kappa_1=1$ by a suitable normalization. In the rest of the paper, we consider $\kappa_1=1$ and $\kappa_2=\kappa$.

This model has been studied in various regimes of the parameter $\kappa$. When $\kappa$ is a constant, \autocite{cipriani2021scaling} showed that its scaling limit coincides with that of the DGFF. Numerical investigations of the two-dimensional integer-valued model with domain-size dependent coefficients were carried out in \autocite{ruiz2005phase}. The scaling limit in this setting was established in \autocite{cipriani2020scaling}, where a phase transition occurs at $\kappa \asymp N^2$. We show that $\kappa\asymp N^2$ also marks a transition in the covariance behavior.

We now give a precise formulation of the model. We always assume the two-dimensional case, and we let $\kappa=N^\lambda$. Consider the discrete Dirichlet problem
\begin{equation}\label{eq:cl}
    \begin{aligned}
        (-\Delta+N^\lambda \Delta^2)G_N(x,y)&=\delta_y(x),&\quad \forall x\in V_N,\\
        G_N(x,y)&=0,&\quad \forall x\notin V_N,
    \end{aligned}
\end{equation}
and let $L_N=(-\Delta+N^\lambda\Delta^2)$. Here, the operator $L_N$ is understood as acting on the space of functions vanishing outside $V_N$. This boundary condition is called the \emph{Dirichlet} boundary condition, or the \emph{clamped} boundary condition. The operator $L_N$ is positive definite (see \autocite{kurt2009maximum}, or see \eqref{eq:posdef} below), and we can consider a Gaussian interface on $V_N$ with zero boundary condition, whose covariance is given by $(L_N)^{-1}=G_N$. In other words, the probability measure $\mathbb{P}_N$ on $\mathbb{R}^{\mathbb{Z}^d}$ has the following density formula:
\begin{equation}\label{eq:probmeasure}
    \mathbb{P}_N(\mathrm{d}\phi)=\frac{1}{Z_N}\exp\left(-\frac{1}{2}\langle \phi, L_N \phi\rangle\right)\prod_{x\in V_N}\mathrm{d}\phi_x\prod_{x\in\mathbb{Z}^d\setminus V_N}\delta_0(\mathrm{d}\phi_x),
\end{equation}
where $Z_N$ is the normalizing constant, called the \emph{partition function}. This model is called a \emph{semiflexible membrane model}.

This boundary condition fixes the height of the random interface to zero outside $V_N$. It also preserves the whole space operator $-\Delta+N^\lambda \Delta^2$ when the field is restricted to $V_N$. On the other hand, this boundary condition makes the analysis more difficult: the restricted bi-Laplacian $\Delta^2_N$ is not equal to $(\Delta_N)^2$, and the eigenbasis for $(-\Delta_N)^2$ does not diagonalize $\Delta^2_N$. In contrast, the same eigenbasis diagonalizes both $-\Delta_N$ and $(\Delta_N)^2$; see, e.g., \eqref{eq:eigen}.

Moreover, the Dirichlet boundary condition preserves the Gibbsian structure associated with the original whole space operator. This provides another reason for taking $\Delta^2_N$, rather than $(\Delta_N)^2$, as the bi-Laplacian term of the original model. The corresponding distinction was already noted in \autocite{kurt2009maximum} for the four-dimensional membrane model.

As already mentioned, $\lambda=2$ gives a critical behavior for the scaling limit. Interestingly, the same value $\lambda=2$ also marks a transition in the large-volume covariance asymptotics. 
More precisely, the following nontrivial behavior occurs. Here, $\operatorname{Cov}(\phi_x,\phi_y)$ is the covariance of $\phi_x$ and $\phi_y$ under the probability measure $\mathbb{P}_N$ given in \eqref{eq:probmeasure}. 
\begin{theorem}[Covariance calculation for competition regime]\label{thm:compregime}
    For $d=2$, let $\{\phi_x\}_{x\in V_N}$ be the centered Gaussian interface, zero outside $V_N$, whose covariance is $G_N$. For each fixed $\lambda\in[0,2]$ and $\delta\in(0,1)$, the following holds for $x,y\in V_N^\delta\coloneq V_{\lfloor (1-\delta)N\rfloor}$: As $N\to\infty$,
    \begin{enumerate}[label=(\alph*)]
        \item(Long-range covariance) If $\|x-y\|\gtrsim N^{\lambda/2}$, then
        \[\operatorname{Cov}(\phi_x,\phi_y)=\operatorname{Cov}^\nabla_N(\phi_x,\phi_y)+O(1).\]
        Here, $\operatorname{Cov}^\nabla_N$ denotes the covariance of the two-dimensional DGFF on $V_N$, with zero boundary condition.
        \item(Short-range covariance) If $\|x-y\|\ll N^{\lambda/2}$, then 
        \[\operatorname{Cov}(\phi_x,\phi_y)=\frac{2-\lambda}{\pi}\log N+O(1).\]
    \end{enumerate}
    The asymptotic notation such as $\gtrsim$ or $\ll$ is defined later; see \Cref{ch:notation}.
\end{theorem}
\begin{remark}
    The well-known asymptotic for $\operatorname{Cov}^\nabla_N(\phi_x,\phi_y)$ on the set $V_N^\delta$ is as follows (see, for example, \autocite[Theorem 1.6.7]{lawler2013intersections}):
    \begin{equation}\label{eq:bulkasymp}
        \operatorname{Cov}^\nabla_N(x,y)=\frac{2}{\pi}\log \left(\frac{N}{\|x-y\|+ 1}\right)+O(1).
    \end{equation}
    Hence, if $\|x-y\|\asymp N$, we have $\operatorname{Cov}^\nabla_N(x,y)=O(1)$. When $N^{\lambda/2}\lesssim \|x-y\|\ll N$, \Cref{thm:compregime} implies that the covariance agrees with the DGFF covariance up to an $O(1)$ error, which is relatively small compared to $\log N$. On the other hand, when $\|x-y\|\asymp N$, the covariance itself remains bounded. Thus the logarithmic structure is most visible in the mesoscopic and microscopic regimes $\|x-y\|\ll N$.
\end{remark}
\begin{remark}
    In \eqref{eq:bulkasymp}, the dependence on the distance $\|x-y\|$ is explicitly reflected. Hence, in \Cref{thm:compregime}, the macroscopic or mesoscopic regimes $\|x-y\|\gtrsim N^{\lambda/2}$ preserve their distance scale. In contrast, when the distance scale is microscopic, i.e., $\|x-y\|\ll N^{\lambda/2}$, all terms involving $\|x-y\|$ are absorbed into the $O(1)$ error, leading to a microscopic distance-information collapse.
\end{remark}

On the other hand, if $\lambda\notin[0,2]$, then either the DGFF or the MM dominates the whole structure, and hence the following covariance estimates are established:
\begin{theorem}[Covariance calculation for dominant regime]\label{thm:domregime}
    For $d=2$, let $\{\phi_x\}_{x\in V_N}$ be as in \Cref{thm:compregime}. Fix $\delta\in(0,1)$. Then the following holds for $x,y\in V_N^\delta$:
    \begin{enumerate}[label=(\alph*)]
        \item(DGFF domination regime) If $\lambda<0$, then
            \[\operatorname{Cov}(\phi_x,\phi_y)=\operatorname{Cov}^\nabla_N(\phi_x,\phi_y)+O(N^\lambda).\]
        \item(MM domination regime) If $\lambda>2$, then
        \[\operatorname{Cov}(\phi_x,\phi_y)=N^{-\lambda}\operatorname{Cov}^{\Delta}_N(\phi_x,\phi_y)+O(N^{4-2\lambda}\log N).\]
    \end{enumerate}
    Here, $\operatorname{Cov}^{\Delta}_N(\phi_x,\phi_y)$ denotes the covariance of the two-dimensional MM with zero boundary conditions. It coincides with the Green function $G^{\Delta}_N(x,y)$ of the following boundary value problem: for any fixed $y\in V_N$,
    \begin{equation}\label{eq:membrane}
        \begin{aligned}
            \Delta^2 G^{\Delta}_N(x,y)&=\delta_y(x),&&\quad \forall x\in V_N,\\
            G^{\Delta}_N(x,y)&=0,&&\quad \forall x\in\partial_2 V_N,
        \end{aligned}
    \end{equation}
    where $\partial_2 V_N\coloneq\{x\notin V_N:\operatorname{dist}(x,V_N)\le 2\}$; see \Cref{ch:notation}.
\end{theorem}

\begin{figure}[!ht]
    \centering
    \begin{subfigure}{\linewidth}
        \includegraphics[width=\linewidth]{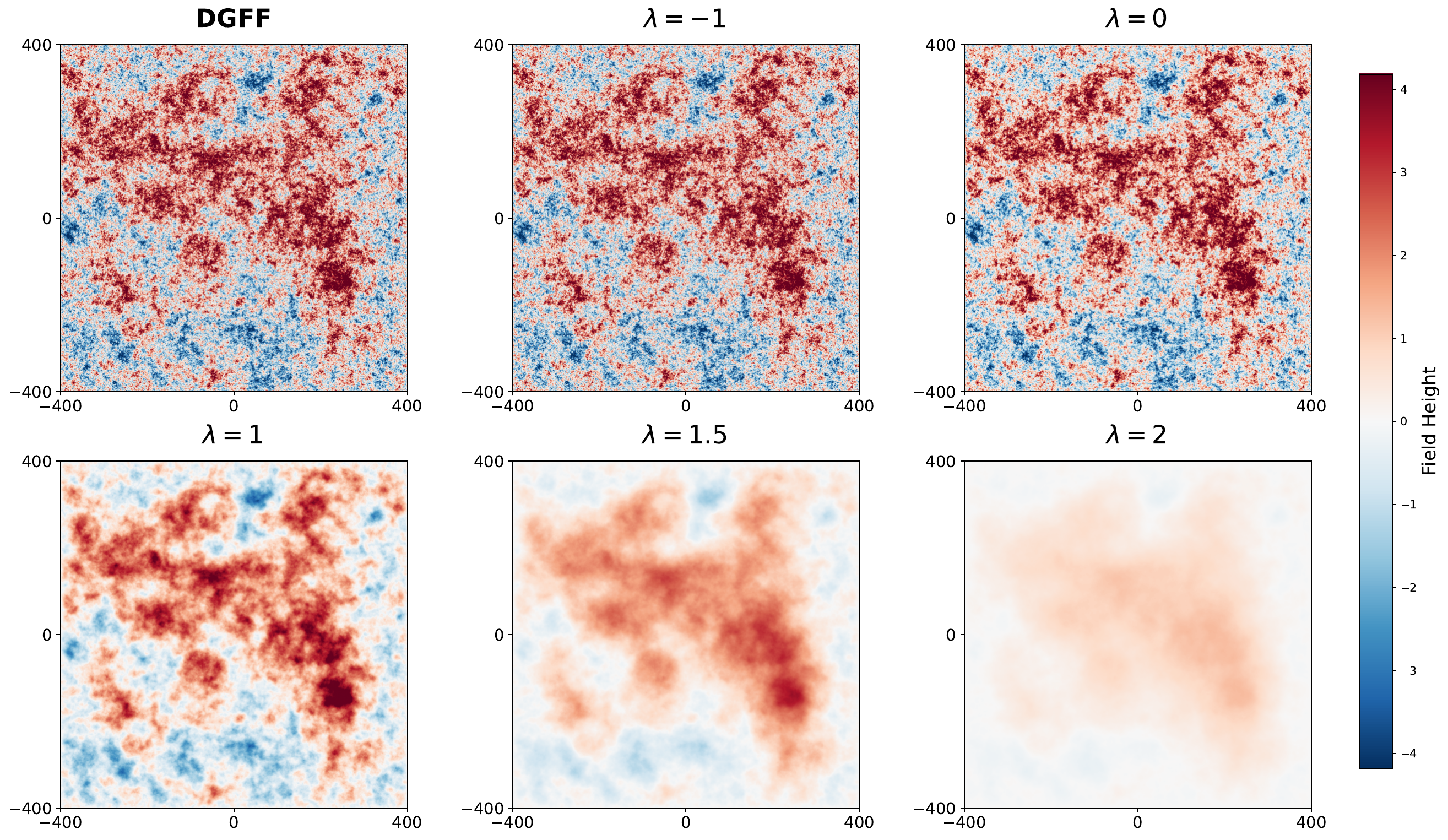}
        \caption{Contour plots of various $\lambda$.}
    \end{subfigure}
    
    \medskip
    
    \begin{subfigure}{\linewidth}
        \includegraphics[width=\linewidth]{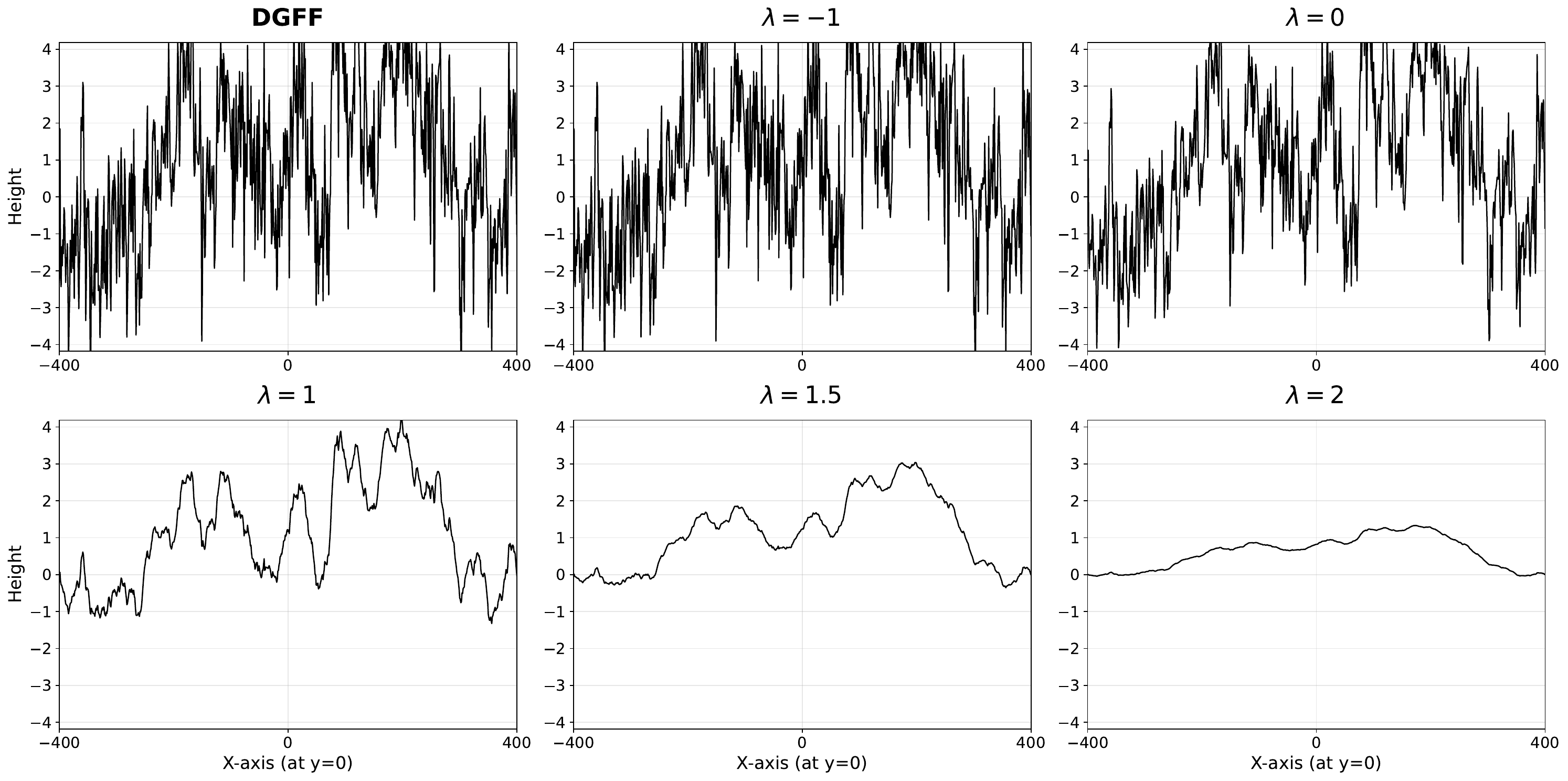}
        \caption{Height profiles along the $x$-axis, corresponding with the contour plots.}
    \end{subfigure}
    \caption{Visualized simulation results for $N=400$ with the same random seed.}
    \label{fig:simulation}
\end{figure}

\Cref{fig:simulation} displays the simulation results of this model with a fixed random seed and varying values of $\lambda\le 2$. As $\lambda$ approaches $2$ from below, the interface becomes visibly smoother, reflecting the increasing influence of the higher-order term.\footnote{We thank an anonymous referee for bringing to our attention the simulation method described in \autocite[Section 4.4]{sheffield2007gaussian}, which we adapted to generate realizations of the semiflexible membrane model with the Dirichlet boundary condition considered here.}

\subsection{Heuristic Motivation}\label{ch:heuristic}
We allow the bending rigidity to depend on the system size and use the polynomial parametrization $\kappa=\kappa_N=N^\lambda$ for convenience in distinguishing the different scaling regimes. As the heuristics below indicate, the relevant crossover scales are indeed polynomial in $N$; in particular, \Cref{thm:compregime} and \Cref{thm:domregime} state that $\lambda=0$ and $\lambda=2$ are the critical thresholds. We illustrate this phenomenon heuristically.

For each $k=(k_1,\dots,k_d)\in\{1,\dots,2N+1\}^d$, the discrete Laplacian $-\Delta_N$ has eigenvalues $\widehat{D}(k)$ and corresponding eigenfunctions $\varphi_k(x)$, given by
\begin{equation}\label{eq:eigen}
    \begin{aligned}
        \widehat{D}(k)&=\frac{2}{d}\sum_{j=1}^d \sin^2 \left(\frac{k_j \pi}{2((2N+1)+1)}\right),\\
        \varphi_k(x)&=\prod_{j=1}^d \sqrt{\frac{1}{N+1}}\sin\left(\frac{k_j \pi (x_j+N+1)}{(2N+1)+1}\right).
    \end{aligned}
\end{equation}
Hence, the eigenvalues are of polynomial order, ranging from $O(N^{-2})$ to $O(1)$. In particular, this suggests that $\lambda<0$ and $\lambda>2$ yield different regimes due to the scaling of the eigenvalues.

To illustrate the spectral mechanism, we temporarily assume that $\Delta^2_N=(\Delta_N)^2$. Since the two operators differ only near the boundary, this simplification captures the expected bulk behavior and motivates the following spectral calculation. Under this assumption, we examine the eigenvalue structure of $(-\Delta+\kappa\Delta^2)^{-1}\simeq (-\Delta)^{-1}(1-\kappa\Delta)^{-1}$.

If $\kappa\ll 1$ (equivalently, $\lambda<0$), then the eigenvalues of $-\kappa\Delta$ are at most $\kappa\ll1$. Hence, all eigenvalues of $(1-\kappa\Delta)^{-1}$ are approximately $1$. Then the only effective eigenvalues come from $(-\Delta)^{-1}$, and this implies that the variance structure is almost the same as that of the DGFF. 

Using the same heuristic argument, when $\kappa \gg N^2$ (equivalently, $\lambda > 2$), the smallest eigenvalue of $\kappa\Delta$ diverges. This suggests the approximation $(1 - \kappa\Delta)^{-1} \approx (-\kappa\Delta)^{-1}$, leading formally to a structure of the form $\kappa^{-1}(\Delta^2)^{-1}$, i.e., a rescaled membrane model.

In contrast, when $1\lesssim \kappa\lesssim N^2$ (equivalently, $\lambda\in[0,2]$), no single term dominates the whole spectrum. Instead, we consider a partial fraction decomposition
\[(-\Delta+\kappa\Delta^2)^{-1}=\frac{\kappa}{(-\kappa\Delta)(1-\kappa\Delta)}=\frac{1}{-\Delta}-\frac{1}{\kappa^{-1}-\Delta}.\]
This decomposition shows a competition structure between $(-\Delta)^{-1}$ and $(1-\kappa\Delta)^{-1}$. We therefore refer to this as the \emph{competition regime}. The behavior of $-\Delta$ is well understood through the DGFF, and we can use an approach from \autocite{michta2022asymptotic} for the massive Green function. 

Moreover, this heuristic explains why we are interested in the case $d=2$. The heuristic above suggests that the Green function $(-\Delta+\kappa\Delta^2)^{-1}$ interpolates between the DGFF and the rescaled membrane model as the scale of $\kappa$ varies from the DGFF-dominated regime to the MM-dominated regime. Since logarithmic correlations arise for the DGFF in dimension $d=2$ and for the membrane model in dimension $d=4$, these two dimensions are natural candidates for logarithmic behavior. However, in the MM-dominated regime $\kappa\gg N^2$, the Green function is approximately $\kappa^{-1}$ times that of the MM. Thus, although logarithmic correlations arise for the MM in dimension $d=4$, the additional normalization $\kappa^{-1}$ suppresses the corresponding logarithmic growth in the present scaling. Therefore, the logarithmic behavior in dimension $d=2$ is the one relevant to the nontrivial interpolation considered here.
\begin{remark}
    The relation between $\|x-y\|$ and $N^{\lambda/2}$ comes from $(\kappa^{-1}-\Delta)^{-1}$ term in the decomposition. The fundamental solution to the screened poisson equation $(m^2-\Delta)u=f$ is given by the \emph{Yukawa potential}. In the physics literature, the parameter $m$ is interpreted as the mass. The Yukawa potential decays like $\exp(-m\|x\|)/\|x\|^\alpha$ for some $\alpha=\alpha(d)$, as $\|x\|\to\infty$. This potential is characterized by a correlation length $\xi\sim 1/m$; specifically, the solution exhibits exponential decay when $|x|\gg\xi$, whereas it behaves qualitatively differently when $|x|\ll\xi$. In our model, $N^{-\lambda/2}$ plays the role of the mass $m$, identifying $N^{\lambda/2}$ as the distance threshold for the correlation decay.    
\end{remark}

\begin{remark}
    There is another heuristic way to understand the scaling in this model. For a function $f\in C^2$, we have
    \begin{gather*}
        f\left(\frac{x+e_i}{N}\right)-f\left(\frac{x}{N}\right)\approx \frac{1}{N}\partial_i f\left(\frac{x}{N}\right),\\
        f\left(\frac{x+e_i}{N}\right)+f\left(\frac{x-e_i}{N}\right)-2f\left(\frac{x}{N}\right)\approx \frac{1}{N^2}\partial_i^2 f\left(\frac{x}{N}\right).
    \end{gather*}
    For a $C^2$ function $f$ with $\phi_x\approx f(x/N)$, we may regard the sums in the Hamiltonian as Riemann sums and obtain
    \[H[\phi]\approx N^{d-2}\int |\nabla f|^2+\kappa N^{d-4}\int |\Delta f|^2.\]
    The two terms are of the same order precisely when $\kappa$ is of order $N^2$.
\end{remark}

\subsection{Proof Strategy}\label{ch:strategy}
As already mentioned, the two finite-volume realizations of the bi-Laplacian $\Delta^2_N$ and $(\Delta_N)^2$ are different, and the former makes the analysis considerably more complicated. To overcome this obstacle, we consider an auxiliary Green function which uses $(\Delta_N)^2$ instead of $\Delta^2_N$. Then we can diagonalize $-\Delta_N$ and $(\Delta_N)^2$ simultaneously, and the arguments in \Cref{ch:heuristic}, including the partial fraction decomposition, can be justified.

We explain this in more detail. Since all eigenvalues of $-\Delta_N$ are positive, the matrix $-\Delta_N+N^\lambda (-\Delta_N)^2$ is invertible. We denote its inverse by $\bar{G}_N$. Equivalently, for each fixed $y\in V_N$, $\bar{G}_N$ solves 
\begin{equation}\label{eq:ss}
    \begin{aligned}
        (-\Delta_N+N^\lambda (-\Delta_N)^2)\bar{G}_N(x,y)&=\delta_y(x),&&\quad \forall x\in V_N,\\
        \bar{G}_N(x,y)&=0,&&\quad\forall x\in V_{N+1}\setminus V_N,\\
        \Delta \bar{G}_N(x,y)&=0,&&\quad \forall x\in\partial V_N,
    \end{aligned}
\end{equation}
where $\partial V_N$ denotes the outer boundary $\{x\notin V_N:\operatorname{dist}(x,V_N)=1\}$; see \Cref{ch:notation}. Here the condition $\Delta\bar G_N=0$ on $\partial V_N$ is understood
in the full-lattice sense. More precisely, let $z\in\partial V_N$ and let $z^-$ be its unique nearest neighbor in $V_N$. Let $z^+\coloneq 2z-z^-$ be the reflection point of $z^-$ across $z$, and we \emph{define} $\bar{G}_N(z^+,y)$ as
\[\bar{G}_N(z^+,y)\coloneq -\bar{G}_N(z^-,y),\]
and set the remaining exterior values as zero. Then
\[\Delta\bar{G}_N(z,y)=\frac{1}{2d}\bigl[\bar{G}_N(z^-,y)+\bar{G}_N(z^+,y)\bigr]=0.\]
This extension realizes the boundary condition required by $(-\Delta_N)^2$.

Under \eqref{eq:ss}, the basis for \emph{discrete sine transform} (DST) given in \eqref{eq:eigen} diagonalizes the operators $-\Delta_N$ and $(-\Delta_N)^2$ simultaneously, so that the spectral calculation becomes rigorous.

In more detail, we use the following argument for the regime $\lambda\in[0,2]$: For each fixed $y\in V_N$, the sine transform $\widehat{\bar{G}}_N$ of the function $\bar{G}_N(\cdot,y)$ follows the identity
\begin{equation}\label{eq:partialfrac}
    \begin{gathered}
        \widehat{\bar{G}}_N(k)=\frac{\widehat{\delta}_y(k)}{\widehat{D}(k)+N^\lambda \widehat{D}(k)^2}=\widehat{\delta}_y(k) \left(\frac{1}{\widehat{D}(k)}-\frac{1}{N^{-\lambda}+\widehat{D}(k)}\right),\\
        \bar{G}_N(x,y)=(-\Delta_N)^{-1}(x,y)-(N^{-\lambda}-\Delta_N)^{-1}(x,y).
    \end{gathered}
\end{equation}
The equations \eqref{eq:partialfrac} have the same structure as the one obtained in the heuristic discussion. Both terms are well studied. The first term is the covariance kernel of the DGFF and admits a classical random-walk representation. The second term admits a representation in terms of a simple random walk killed independently at each step with probability $N^{-\lambda}/(1+N^{-\lambda})$. We discuss random walks with killing rate in \Cref{ch:randomwalk}. After analyzing $\bar{G}_N$ via \eqref{eq:partialfrac}, we can derive the estimate
\[\sup_{x,y\in V_N^\delta}|\bar{G}_N(x,y)-G_N(x,y)|\le C_{\delta,\lambda},\]
using a Sobolev embedding argument similar to that used in \autocite{kurt2009maximum}.

In the regime $\lambda<0$, we use the fact that the operators $\bar G_N$ and $(-\Delta_N)^2$ share the same eigenbasis given in \eqref{eq:eigen}. We estimate $\bar{G}_N$ first by using this spectral observation. After estimating $\bar{G}_N$, we control the pointwise difference of two Green functions $G_N(x,y)$ and $\bar{G}_N(x,y)$ by the diagonal variance scale $\bar{G}_N(x,x)$, uniformly for $x,y\in V_N^\delta$. This proof is technically simpler than that of the case $\lambda\in[0,2]$, since it avoids the more delicate analytic tools, such as Sobolev embedding. It uses the resolvent identity and the standard estimate of $(-\Delta_N)^{-1}(x,y)$.

In contrast, we do not use the auxiliary Green function $\bar{G}_N$ when $\lambda>2$. Such a comparison between $G_N$ and $\bar{G}_N$ would require complicated componentwise estimates for $\Delta^2_N$, and this comparison is unnecessary in this regime. Instead, we directly compare $G_N$ and $N^{-\lambda}G^\Delta_N$. Here the key ingredients are the Neumann series representation for the matrix $G_N$, together with estimates for $G^\Delta_N=(\Delta^2_N)^{-1}$ obtained in \autocite{muller2019estimates}.

\subsection{Further Directions}\label{ch:futurework}
We outline several possible directions for future research. These can be broadly divided according to the behavior of the parameter $\lambda$: the global regime with fixed $\lambda$, the near-critical regime where $\lambda$ depends on $N$, and the critical regime $\lambda = 2$.

\paragraph{Global regime: behavior of the maximum} 
As already mentioned, the DGFF for $d=2$ and the MM for $d=4$ are log-correlated Gaussian fields. An important question for log-correlated Gaussian fields concerns the behavior of their maxima. The maximum of the two-dimensional DGFF has been studied, for example, in \autocite{bolthausen2001entropic,bramson2012tightness,ding2014extreme,bramson2016convergence}, and it has been shown that
\[\operatorname{Law}\left(\max_{x\in V_N}\phi_x^\text{DGFF}-\left(2\log N-\frac{3}{4}\log\log N\right)\right)\Longrightarrow \mu_\infty^\text{DGFF},\]
up to the normalization of the field. Here $\mu_\infty^\text{DGFF}$ is a nondegenerate probability distribution. Likewise, it has been proven that the centered maximum of the four-dimensional MM follows
\[\operatorname{Law}\left(\max_{x\in V_N}\phi_x^\text{MM}-\left(2\sqrt{2}\log N-\frac{3}{4\sqrt{2}}\log\log N\right)\right)\Longrightarrow \mu_\infty^\text{MM},\]
see, for example, \autocite{kurt2009maximum,schweiger2020maximum}.

\Cref{thm:compregime} means that the semiflexible membrane with $\lambda\in[0,2]$ is log-correlated only at distances above the crossover scale. Motivated by this long range scale log-correlated structure, we conjecture that the maximum of the semiflexible membrane satisfies that
\[\max_{x\in V_N}\phi_x-\left(c_1(\lambda)\log N-c_2(\lambda)\log\log N\right)\]
is tight, when $\lambda\in[0,2)$. However, we cannot directly apply the previous arguments from those pioneering works, as the semiflexible membrane does not have log-correlation at short distances, and it might be a challenging and subtle problem that we leave for future work.

This conjecture is also motivated by existing results for related scale-depend\allowbreak{}ent log-correlated Gaussian fields. For example, \autocite{acosta2014tightness} studied the tightness of the recentered maximum of a continuum  Gaussian field $\{Y_x^\varepsilon\}_{x\in[0,1]^d}$, satisfying the covariance estimate
\[|\operatorname{Cov}(Y_x^\varepsilon,Y_y^\varepsilon)+\log(\max(\varepsilon,\|x-y\|))|\le C_Y,\]
and
\[\mathbb{E}|Y_x^\varepsilon-Y_y^\varepsilon|^2\le C_Y \varepsilon^{-1}\|x-y\|\quad\text{for }\|x-y\|\le \varepsilon.\]
Under these assumptions, 
\[\max_{x\in[0,1]^d}Y_x^\varepsilon-\left(\sqrt{2d}\log \varepsilon^{-1}-\frac{3}{2\sqrt{2d}}\log\log\varepsilon^{-1}\right)\]
is tight. Moreover, \autocite{acosta2016convergence} established the convergence in law of the centered maximum when $Y_x^\varepsilon$ is a mollified Gaussian free field, a specific example satisfying \autocite{acosta2014tightness}. We leave the study of the maximum of the semiflexible membrane for future work.

On the other hand, when $\lambda<0$, the covariance is close to that of the DGFF with a negligible error, and it is expected that the arguments in \autocite{bramson2016convergence} can be directly applied. When $\lambda>2$, however, the covariance is close to that of the \emph{two-dimensional} rescaled MM with a small error, and the model does not follows log-correlation. Hence, qualitatively different arguments would be required.

\paragraph{Near-critical regime: critical window} We characterize the full asymptotic structure of the model for each \emph{fixed} $\lambda\in\mathbb{R}$, and the values $\lambda=0$ and $\lambda=2$ correspond to the critical thresholds. A natural question is to understand the near-critical behavior, where $\lambda_N=\lambda+\varepsilon_N$ depends on $N$ and $\varepsilon_N\to0$. In particular, identifying the appropriate decay rate of $\varepsilon_N$ is of central interest.

\paragraph{Critical regime: limit object} For the critical mass regime $\lambda=2$, the variance is uniformly bounded in $N$. It is therefore natural to ask about the limiting behavior of the corresponding fields as $N\to\infty$, and in particular to investigate possible limiting objects on $\mathbb{Z}^2$. We leave this question for future work.

\subsection{Notation}\label{ch:notation}
In this section, we introduce the notation used throughout the paper. We introduce the notation for general $d\ge1$, although our main focus is the case $d=2$.

\paragraph{General setting}
Throughout the paper, $c,C,\dots>0$ denote generic constants whose values may vary from line to line. They may depend on some fixed parameters, but they never depend on $N$, $x$, and $y$ unless explicitly stated otherwise. 

\paragraph{Green functions}
The notation $G_N(x,y)$ and $\bar{G}_N(x,y)$ denote the Green functions associated with \eqref{eq:cl} and \eqref{eq:ss}, respectively. Additional notation related to Green functions will be introduced in \Cref{ch:randomwalk}. Recall that $\Delta_N^2$ denotes the restriction of the whole-space bi-Laplacian, and is not the square of $\Delta_N$.

Whenever no confusion can arise, we identify covariance kernels with the corresponding Green functions. In particular,
\[\operatorname{Cov}(\phi_x,\phi_y)=G_N(x,y),\]
and analogous identifications are used for the DGFF and MM covariance kernels.

\paragraph{Lattice objects}
For $x\in\mathbb{Z}^d$, the notation $\|x\|$ denotes the Euclidean norm. Sometimes we use $\|x\|_1$ to denote its $\ell^1$ norm. Note that the metric induced by $\|\cdot\|_1$ is the graph distance on $\mathbb{Z}^d$, denoted as $\operatorname{dist}(\cdot,\cdot)$.

Throughout, the boundary of a region is understood as the outer boundary layer. For $k\in\mathbb{N}$, we define the outer boundary 
\[\partial_k V_N\coloneq \{x\notin V_N:\operatorname{dist}(x, V_N)\le k\}.\]
We sometimes use the inner boundary layer
\[\partial_k^- V_N\coloneq \{x\in V_N:\operatorname{dist}(x, V_N^c)\le k\}.\]
If $k=1$, then we simply write $\partial V_N$ and $\partial^- V_N$.

We are mainly interested in the behavior away from the boundary. For a constant $\delta\in(0,1)$, we focus on the set $V_N^\delta\coloneq V_{\lfloor(1-\delta)N\rfloor}$. It is referred to as the $\delta$-bulk, or simply the bulk.

For a directed edge $e$ in $\mathbb{Z}^d$ (i.e., $e$ is a unit vector in $\mathbb{Z}^d$), $\nabla_e f(x)\coloneq f(x+e)-f(x)$ denotes the discrete derivative in direction $e$, and
\begin{gather*}
    \nabla f\cdot \nabla g=\nabla f(x)\cdot \nabla g(x)\coloneq \frac{1}{2d}\sum_e \nabla_e f(x)\nabla_e g(x),\\
    \|\nabla f\|_{L^2(\Lambda)}\coloneq \left(\frac{1}{2d}\sum_e\sum_{x\in\Lambda}(\nabla_e f(x))^2\right)^{1/2}.
\end{gather*}
The normalization factor $1/2d$ ensures a convenient summation-by-parts formula. We also use higher-order finite differences $\nabla^\alpha$, where $\alpha$ is a multiindex. To unify forward and backward differences, let $\mathcal{I}_d\coloneq \{\pm1,\dots,\pm d\}$, and $e_{-j}\coloneq -e_j$. For $\alpha=(\alpha_j)_{j\in\mathcal{I}_d}\in\mathbb{N}^{\mathcal{I}_d}$ (with $\mathbb{N}=\{0,1,2,\dots\}$), define
\[\nabla^\alpha\coloneq \prod_{j\in\mathcal{I}_d} \nabla_{e_j}^{\alpha_j}.\]

\paragraph{Asymptotic notation}
Let $A$ and $B$ be nonzero functions of $N$ (and possibly other variables).
We write $A\ll B$, $B\gg A$ or equivalently $A=o(B)$, if $|A/B|\to 0$. 
If $\limsup |A/B| <\infty$, we write $A=O(B)$, $A\lesssim B$, or $B\gtrsim A$.
If $A=O(B)$ and $B=O(A)$, then we write $A\asymp B$. 
Moreover, if $\lim A/B=1$, we write $A\sim B$ (or $A\simeq B$). 

Unless otherwise specified, asymptotic notation with respect to $N$ is understood to be uniform in $x \in V_N^\delta$, for a fixed $\delta\in(0,1)$. When relevant, it is also uniform in $y \in V_N^\delta$.
    % !TEX root = ../main.tex
\section{Random Walks and Green Functions}\label{ch:randomwalk}
In this section, we briefly summarize some properties of random walks, which yield useful tools for the analysis of random interfaces. For more detail, see \autocite[Chapter 8]{friedli2017statistical} and \autocite{lawler2013intersections}.

\subsection{Random Walks and Random Interfaces}\label{ch:RWintro}
It is well known that random walks are closely related to random interface models such as the DGFF. The Gaussian free field on a finite subset $\Lambda\subset\mathbb{Z}^d$ with zero boundary condition is a centered Gaussian field with covariance matrix $(-\Delta_\Lambda)^{-1}$. That is, the covariance matrix $G^0_\Lambda$ of GFF satisfies
\begin{align*}
    (-\Delta_\Lambda)G^0_\Lambda(x,y)&=\delta_y(x), x\in\Lambda,\\
    G^0_\Lambda(x,y)&=0, x\in\partial\Lambda.
\end{align*}

Let $S_n=S_0+X_1+\dots+X_n$ be the symmetric random walk on $\mathbb{Z}^d$, starting at the point $S_0\in \mathbb{Z}^d$. That is, each $X_j$ is independent and follows the uniform distribution on the set $\{\pm e_1,\dots,\pm e_d\}$. From the random walk viewpoint, the Green function $G^0_\Lambda(x,y)$ has the following representation 
\[G^0_\Lambda(x,y)=\mathbb{E}^x\left[\sum_{n=0}^{\tau_\Lambda-1}\mathbbm{1}\{S_n=y\}\right],\]
where $\tau_\Lambda:=\inf\{n:S_n\notin\Lambda\}$ is the first exit time from $\Lambda$.

A useful generalization is the massive Laplacian operator $m^2-\Delta_\Lambda$, and its Green function $G^m_\Lambda$ admits a representation in terms of a symmetric random walk $\tilde{S}_n$ with independent killing rate $m^2/(1+m^2)$. 
That is, the random walk $\tilde{S}_n$ has a ``cemetery'' state $\dagger$, and for each step, the random walk $\tilde{S}_n$ jumps to $\dagger$ with probability $m^2/(1+m^2)$, independently of the past. Once it enters the state $\dagger$, it stays there forever. In other words,
\[\mathbb{P}(\tilde{S}_{n+1}=\dagger|\tilde{S}_n\in\mathbb{Z}^d)=\frac{m^2}{1+m^2},\quad \mathbb{P}(\tilde{S}_{n+1}=\dagger|\tilde{S}_n=\dagger)=1.\]
For computational convenience, we define $G^m(\dagger,\cdot)=G^m_\Lambda(\dagger,\cdot)=0$. In many physical contexts, $m$ is interpreted as ``mass''. The Green function $G^m$ also has a similar expectation formula
\[G^m_\Lambda(x,y)=\frac{1}{1+m^2}\mathbb{E}^x\left[\sum_{n=0}^{\tilde{\tau}_\Lambda-1}\mathbbm{1}\{\tilde{S}_n=y\}\right],\]
where $\tilde{\tau}_\Lambda:=\inf\{n:\tilde{S}_n\in\mathbb{Z}^d\setminus\Lambda\}$ still denotes the first exit time \emph{before the walk is killed}.

Using the standard simple random walk $S_n$, the Green function $G^m$ for the operator $m^2-\Delta_\Lambda$ can be written as
\begin{equation}\label{eq:massiveboundeddomain}
    G^m_\Lambda(x,y)=\mathbb{E}^x\left[\sum_{n=0}^{\tau_\Lambda-1}(1+m^2)^{-(n+1)}\mathbbm{1}\{S_n=y\}\right],
\end{equation}
where $\tau_\Lambda$ is the exit time of the standard random walk.

When $\Lambda\uparrow\mathbb{Z}^d$, the existence of the infinite-volume limit of the massless Green function depends on the dimension $d$. If $d\ge3$, the limit exists, and the walk is said to be transient. When $d<3$, however, the limit does not converge, and the walk is recurrent. That is, the random walk visits every point infinitely many times. To avoid this issue, one considers the potential kernel
\begin{equation}\label{eq:potker}
    \mathfrak{a}(x):=\lim_{\Lambda\uparrow\mathbb{Z}^d}[G_\Lambda(x)-G_\Lambda(0)],
\end{equation}
which does converge. In many contexts, the potential kernel is defined by the limit of $G_\Lambda(0)-G_\Lambda(x)$. In this paper, however, we adopt the opposite sign convention so that the transient and recurrent cases have the same formal expression. On the other hand, the Green function $G^m$ with the constant mass $m>0$ always admits an infinite-volume limit
\begin{equation}\label{eq:massivewholespace}
    G^m(x,y)=\frac{1}{1+m^2}\mathbb{E}^x\left[\sum_{n=0}^{\infty}\mathbbm{1}\{\tilde{S}_n=y\}\right]=\mathbb{E}^x\left[\sum_{n=0}^{\infty}(1+m^2)^{-(n+1)} \mathbbm{1}\{S_n=y\}\right],
\end{equation}
regardless of the dimension.

\begin{remark}
    An additional survival rate factor appears in \eqref{eq:massiveboundeddomain} and \eqref{eq:massivewholespace}. This discrepancy comes from different conventions for the term ``Green function''. In the random walk literature (e.g., \autocite{lawler2013intersections}), the Green function usually refers to the expected number of visits of the random walk. In contrast, in the literature on random interfaces (e.g., \autocite{friedli2017statistical}), the Green function is typically defined as the covariance kernel of the field. These two conventions coincide only when the survival rate equals $1$. The derivation of the additional factor is included in \autocite{friedli2017statistical}. In this paper, we follow the latter convention.
\end{remark}
\paragraph{Notational remarks for Green functions}
In the rest of this paper, we refer to the Green function of $-\Delta$ or $-\Delta_\Lambda$ as the \emph{massless Green function}, and similarly to the Green function of $(m^2-\Delta)$ and $(m^2-\Delta_\Lambda)$ as the \emph{massive Green function}. 

Sometimes we use the superscript $i$ to denote both massless and massive objects simultaneously. For example, $G^i_\Lambda$ means the Green function on $\Lambda$ corresponding to either the massless or massive Laplacian. If $\Lambda=V_N$, we simply write the subscript $N$ instead of $V_N$.

We further write the survival rate $(1+m^2)^{-1}$ as $\rho$. Although $\rho$ depends on $m$, we suppress this dependence whenever it is clear from the context.

\subsection{Harmonic Measure Representation}\label{ch:rwrep}
The massless Green function $G^0_\Lambda$ admits the following decomposition
\[G^0_\Lambda(x,y)=G^0(x,y)-\sum_{z\in\partial\Lambda}H^0_\Lambda(x,z)G^0(z,y),\]
or, by symmetry, we can interchange the order of $x$ and $y$. When $d<3$, we interpret $G^0(x,y)$ in this representation as the potential kernel $\mathfrak{a}(x-y)$ defined in \eqref{eq:potker}. Here,
\[H^0_\Lambda(x,z)\coloneq \mathbb{P}^x (S_{\tau_\Lambda}=z)=\mathbb{E}^x[\mathbbm{1}\{S_{\tau_\Lambda}=z\}],\quad \tau_\Lambda:=\inf\{n:S_n\notin \Lambda\}.\]
From its definition, clearly it holds that
\[H^0_\Lambda(x,z)\ge0,\quad \sum_{z\in\partial \Lambda}H^0_\Lambda(x,z)=1,\quad\forall x\in \Lambda.\]
In other words, $H^0_\Lambda(x,\cdot)$ is a probability density on $\partial \Lambda$.

The massive Green function $G^m_N$ also admits an analogous decomposition
\[G^m_\Lambda(x,y)=G^m(x,y)-\sum_{z\in\partial\Lambda}H^m_\Lambda(x,z)G^m(z,y),\]
where $H^m_\Lambda(x,z)$ denotes the probability that a simple random walk with killing rate starting at $x$ exits $\Lambda$ at the point $z$, before it is killed. With the standard random walk notation, it coincides with
\begin{equation}\label{eq:massiveboundary}
    H^m_\Lambda(x,z)=\mathbb{E}^x[\rho^{\tau_\Lambda}\mathbbm{1}\{S_{\tau_\Lambda}=z\}],
\end{equation}
where $\rho$ denotes the survival rate. From its definition, the sum of $H^m_\Lambda$ over all $z\in\partial \Lambda$ is less than $1$ unless $\rho=1$. The proof for the massive case is analogous to that of the massless case. See the proof for the massless case, e.g., \autocite{lawler2013intersections}. 

In summary, we have
\begin{equation}\label{eq:GreenDecomp}
    G^i_\Lambda(x,y)=G^i(x,y)-\sum_{z\in\partial\Lambda}H^i_\Lambda(x,z)G^i(z,y).
\end{equation}

\subsection{Estimates for the Massive Green Functions}
We collect several estimates for the massive Green function in different asymptotic regimes of $m\|x\|$. In \Cref{thm:compregime}, we distinguish the estimates for the covariance according to the relative size of $\|x-y\|$ and $N^{\lambda/2}$. Accordingly, we classify the estimates according to the asymptotic behavior of $m\|x\|$.

We summarize the estimates below. Throughout this subsection, we assume $d=2$. The proof for \Cref{thm:massiveanalysis} is discussed in Appendix \ref{section:massiveanalysis}.
\begin{proposition}[Uniform estimate for massive Green functions]\label[proposition]{thm:massiveanalysis}
    For a fixed $M>0$, let $0<m\le M$ be a mass parameter. Suppose $x\in\mathbb{Z}^2$ satisfies $\|x\|\to\infty$. Let $G^m$ be the massive Green function with mass $m$. Then we have the following estimates:
    \begin{enumerate}[label=(\alph*), ref=(\alph*)]
        \item \label{cond:subcrit}
        If $m\|x\|\to\infty$, then there exists $c,C>0$, depending only on $M$, such that
        \[G^m(x,0)\le C (m\|x\|)^{-1/2}e^{-cm\|x\|}.\]
        
        \item \label{cond:crit}
        If $m\|x\|\to s>0$ with $\|x\|\to\infty$, then for any multiindex $\alpha$ including $\alpha=0$,
        \[|\nabla^\alpha G^m(x,0)|\lesssim \|x\|^{-|\alpha|}.\]
        
        \item \label{cond:supercrit}
        If $m\|x\|\to 0$, then 
        \[|G^m(x,0)-G^m(0,0)-\mathfrak{a}(x)|=o(1).\]
        Moreover, the asymptotic \ref{cond:supercrit} remains valid without the assumption $\|x\|\to\infty$.
    \end{enumerate}
\end{proposition}
    % !TEX root = ../main.tex
\section{Proof of \texorpdfstring{\Cref{thm:compregime}}{Theorem 1.1}}\label{section:proofcompregime}
We prove \Cref{thm:compregime}. The proof is divided into two parts. In the first part, we analyze $\bar{G}_N$ defined in \eqref{eq:ss}. In the second part, we show that $G_N$ is well approximated by $\bar{G}_N$, with the error controlled using a Sobolev embedding argument. 

\subsection{The simply supported problem}\label{ch:SSP}
We treat the long-range regime and the short-range regime separately, since different parts of \Cref{thm:massiveanalysis} are used in the two regimes.

\subsubsection{Long- and Intermediate-range Covariance}
Our goal is to prove the following proposition.
\begin{proposition}\label[proposition]{thm:sslongrange}
    Let $\lambda\in[0,2]$ and $\delta\in (0,1)$ be fixed. For any $x,y\in V_N^\delta$, if $\|x-y\|\gtrsim N^{\lambda/2}$, then 
    \begin{equation}\label{eq:sslongrange}
        \bar{G}_N(x,y)=G^0_N(x,y)+O(1),
    \end{equation}
    uniformly in $x,y\in V_N^\delta$. Moreover, the $O(1)$ error term improves to $o(1)$ if $\|x-y\|\gg N^{\lambda/2}$.
\end{proposition}

\begin{proof}
    We start with the decompositions \eqref{eq:partialfrac} and \eqref{eq:GreenDecomp}. We have 
    \begin{equation}\label{eq:longrangedecomp}
        \bar{G}_N(x,y)=G^0_N(x,y)-G^m(x,y)+\sum_{z\in\partial V_N}H^m_N(x,z)G^m(z,y).
    \end{equation}
    It has the following estimate: since $H^m_N\ge0$ and $G^m\ge0$, 
    \begin{equation}
        \begin{aligned}
            0\le \sum_{z\in\partial V_N}H_N^m(x,z)G^m(z,y)
            &\le \bigl(\sup_{z\in\partial V_N}G^m(z,y)\bigr) \sum_{z\in\partial V_N}H_N^m(x,z)\\
            &\overset{\mathclap{\eqref{eq:massiveboundary}}}{\le} \sup_{z\in \partial V_N}G^m(z,y).
        \end{aligned}
    \end{equation}
    Now we use the translation invariance of the infinite-volume massive Green function 
    \[G^m(x,y)=G^m(x-y,0).\]
    See the beginning of Appendix \ref{section:massiveanalysis}.
    To apply \Cref{thm:massiveanalysis}, we consider the mass and the distance scale. In our case, the mass $m$ corresponds to $N^{-\lambda/2}$, which vanishes as $N\to\infty$ when $\lambda>0$. Since we fix $y\in V_N^\delta$, the estimate $\|z-y\|\asymp N$ holds uniformly in $\partial V_N$. If $\lambda\in(0,2)$, then we apply \Cref{thm:massiveanalysis} \ref{cond:subcrit}. Even when $\lambda=0$, $m\|z-y\|\to\infty$ still holds as $N\to\infty$, and the same argument is adaptable. If $\lambda=2$, then we apply \ref{thm:massiveanalysis} \ref{cond:crit}. Then we have
    \begin{equation}
        \sum_{z\in\partial V_N}H_N^m(x,z)G^m(z,y)=\begin{cases}
            o(1)&\lambda\in[0,2),\\
            O(1)&\lambda=2.
        \end{cases}
    \end{equation}
    Under the assumption $\|x-y\|\gtrsim N^{\lambda/2}$, the same argument shows that $G^m(x,y)$ is uniformly bounded when $\lambda>0$. When $\lambda=0$, the mass is constant, and we use $0\le G^m(x,y)\le O(1)$ from \eqref{eq:massivewholespace}. Then $G^m(x,y)$ can be absorbed into the $O(1)$ error term whether $\lambda>0$ or $\lambda=0$. We get
    \[\bar{G}_N(x,y)=G^0_N(x,y)+O(1),\]
    where $O(1)$ is uniform in $x,y\in V_N^\delta$ and $N$.
\end{proof}

\subsubsection{Short-range Covariance for Two Dimensions}\label{ch:shortrange}
To avoid ambiguity, here we use the potential kernel notation $\mathfrak{a}(x-y)$ instead of $G^0(x,y)$. The goal of this section is to show the following proposition.
\begin{proposition}\label[proposition]{thm:ssshortrange}
    Let $\lambda\in[0,2]$ and $\delta\in (0,1)$ be fixed. For any $x,y\in V_N^\delta$, if $\|x-y\|\ll N^{\lambda/2}$, then
    \begin{equation}\label{eq:ssshortrange}
        \bar{G}_N(x,y)=\frac{2-\lambda}{\pi}\log N+O(1).
    \end{equation}
\end{proposition}
\begin{proof}
    We first consider the case $\lambda=0$. In this case, the assumption $\|x-y\|\ll N^{\lambda/2}=1$ implies that $x=y$ for all sufficiently large $N$. Hence, 
    \[\bar G_N(x,x)=G_N^0(x,x)-G_N^m(x,x)=\frac{2}{\pi}\log N+O(1),\]
    where we used $0\le G_N^m(x,x)\le G^m(0,0)=O(1)$ again.

    We may therefore assume $0<\lambda\le2$. Under the given assumption, the mass $m$ in \Cref{thm:massiveanalysis} corresponds to $N^{-\lambda/2}$. We use \Cref{thm:massiveanalysis} \ref{cond:supercrit} together with \eqref{eq:GreenDecomp}. Using \Cref{thm:massiveanalysis} \ref{cond:supercrit}, if $\|x-y\|\ll N^{\lambda/2}$, then
    \begin{equation}\label{eq:shortmassivedecomp}
        \begin{aligned}
            G^m_N(x,y)&=G^m(x,y)-\sum_z H^m_N(x,z)G^m(z,y)\\
            &=G^m(x,y)-G^m(0,0)-\sum_z H^m_N(x,z)G^m(z,y)+G^m(0,0)\\
            &=\mathfrak{a}(x-y)-\sum_z H^m_N(x,z)G^m(z,y)+G^m(0,0)+o(1).
        \end{aligned}
    \end{equation}
    Subtracting \eqref{eq:shortmassivedecomp} from $G^0_N(x,y)=\mathfrak{a}(x-y)-\sum_z H^0_N(x,z)\mathfrak{a}(z-y)$, the term $\mathfrak{a}(x-y)$ cancels, namely,
    \begin{equation}\label{eq:shortssdecomp}
        \begin{aligned}
            \bar{G}_N(x,y)=&G^0_N(x,y)-G^m_N(x,y)\\
            =&\sum_z [H^m_N(x,z)G^m(z,y)-H^0_N(x,z)\mathfrak{a}(z-y)]-G^m(0,0)+o(1).
        \end{aligned}
    \end{equation}
    We have $\|z-y\|\asymp N$ uniformly in $z\in\partial V_N$, and \Cref{thm:massiveanalysis} \ref{cond:subcrit} and \ref{thm:massiveanalysis} \ref{cond:crit} show that the massive boundary $\sum_z H^m_N(x,z)G^m(z,y)$ contributes $o(1)$ when $\lambda\in[0,2)$, or $O(1)$ when $\lambda=2$.
    
    The contribution for the massless boundary $\sum_z H^0_N(x,z)\mathfrak{a}(z-y)$ is as follows: The potential kernel $\mathfrak{a}$ has the following asymptotic (see, e.g., \autocite[Theorem 1.6.2.]{lawler2013intersections}):
    \[\mathfrak{a}(x)=-\frac{2}{\pi}\log |x|+O(1),\quad \|x\|\to\infty.\]
    The uniform asymptotic $\|z-y\|\asymp N$ yields
    \begin{equation}\label{eq:ss_massiveboundary}
        \begin{aligned}
            -\sum_z H^0_N(x,z)\mathfrak{a}(z-y)&=-\sum_z H^0_N(x,z)\left(-\frac{2}{\pi}\log N+O(1)\right)\\
            &=\frac{2}{\pi}\log N+O(1),
        \end{aligned}
    \end{equation}
    where we used the property $\sum_z H^0_N(x,z)=1$. From \eqref{eq:shortssdecomp} and \eqref{eq:ss_massiveboundary}, we have
    \[G^0_N(x,y)-G^m_N(x,y)=\frac{2}{\pi}\log N-G^m(0,0)+O(1).\]
    
    The massive Green function with the mass $m\to0+$ has an asymptotic $G^m(0,0)\simeq \frac{2}{\pi}|\log m|$ (see, e.g., \autocite[Proposition 8.32]{friedli2017statistical}). It can be improved as $G^m(0,0)=\frac{2}{\pi}|\log m|+O(1)$; see Appendix \ref{ch:originvar}. In our case, the mass $m$ corresponds to $N^{-\lambda/2}$, and this yields the sharp competition structure
    \[\bar{G}_N(x,y)=\frac{2-\lambda}{\pi}\log N+O(1),\]
    which is the desired result in \Cref{thm:ssshortrange}.
\end{proof}
    % !TEX root = ../main.tex
\subsection{Recovering the Boundary Condition}
We have analyzed $\bar{G}_N$ in the previous section. Now we compare $\bar{G}_N$ and $G_N$ by following the setup of \autocite{kurt2009maximum}. More precisely, the goal is to prove the following proposition.
\begin{proposition}\label[proposition]{thm:competecomparison}
    Let $d=2$, and fix $\lambda\in[0,2]$ and $\delta\in(0,1)$. Then the following asymptotic
    \begin{equation}
        \sup_{x,y\in V_N^\delta} |G_N(x,y)-\bar{G}_N(x,y)|=O(1)
    \end{equation}
    holds. 
\end{proposition}
In this section, we fix $\delta\in(0,1)$ unless it is mentioned. The rest of this section is devoted to the proof of \Cref{thm:competecomparison}.

\subsubsection{Norm Estimate of the Difference}
For each fixed $y\in V_N^\delta$, define $H_N(\cdot,y)=\bar{G}_N(\cdot,y)-G_N(\cdot,y)$. From the definition of $G_N$ and $\bar{G}_N$, $H_N$ satisfies 
\begin{equation*}
    \begin{aligned}
        LH_N(x,y)&=0 &&x\in V_N,\\
        H_N(x,y)&=\bar{G}_N(x,y) &&x\in\partial_2 V_N.
    \end{aligned}
\end{equation*}
Here $L$ is the whole space operator $L=-\Delta+N^\lambda \Delta^2$.

Let $f_0$ be a function satisfying $f_0(x)=\bar{G}_N(x,y)$ for $x\notin V_N^{\delta/2}$ and $f_0(x)=0$ on $V_N^{\delta'}$ for some $\delta/2<\delta'<\delta$. For $x\notin V_N$, the value $f_0(x)=\bar{G}_N(x,y)$ is understood as the reflection extension described after \eqref{eq:ss}. We interpolate $f_0$ on $V_N^{\delta/2} \setminus V_N^\delta$ so that 
\begin{equation}\label{eq:interpolation}
    \sup_{x\in V_N^{\delta/2}} |\nabla^\alpha f_0(x)|\lesssim N^{-|\alpha|}
\end{equation}
for any nonzero multiindex $\alpha$. Such a function exists as a smooth interpolation between $0$ and $\bar{G}_N (\cdot,y)$. Here and below, the term \emph{smooth} means that
\[|\nabla^\alpha f| = O(N^{-|\alpha|})\]
for every nonzero multiindex $\alpha$.
\begin{lemma}\label[lemma]{lemma:existence}
    Let $0<\delta,\delta'<1$ be $\delta/2<\delta'<\delta$. There exists a function $f_0$ satisfying $f_0(x)=\bar{G}_N(x,y)$ for $x\notin V_N^{\delta/2}$, $f_0(x)=0$ on $V_N^{\delta'}$, and \eqref{eq:interpolation}.
\end{lemma}
\begin{proof}
    We analyze the behavior of $\bar{G}_N$ on a buffered interpolation subannulus $V_N^{\delta_1}\setminus V_N^{\delta_2}$ of the initial interpolation layer $V_N^{\delta/2}\setminus V_N^{\delta'}$. Here $\delta_1$ and $\delta_2$ satisfy $\delta/2<\delta_1<\delta_2<\delta'<\delta$. It is sufficient to show that such interpolation for $f_0$ occur only on this subannulus. That is, if $f_0$ satisfies $f_0(x)=0$ on $V_N^{\delta_2}$ and $f_0(x)=\bar{G}_N(x,y)$ for $x\notin V_N^{\delta_1}$, then it satisfies the desired interpolation condition.
        
    For any $x\in V_N^{\delta_1}\setminus V_N^{\delta_2}$, $y\in V_N^\delta$ and $z\in \partial V_N$,
    \begin{equation}
        (\delta-\delta_2)N-1\le \|x-y\|,\quad \delta_1 N-1\le \|x-z\|.
    \end{equation}
    That is, $\|x-y\|\asymp N$ and $\|x-z\|\asymp N$ hold, which are \emph{uniform} in $y$ and $z$.
    
    We use the decomposition \eqref{eq:GreenDecomp} on $\bar{G}_N=G^0_N-G^m_N$. When $|\alpha|>0$, the massless part satisfies the estimate
    \[\nabla^\alpha G^0(x,y)=\nabla^\alpha \mathfrak{a}(x-y)=O(\|x-y\|^{-|\alpha|}),\]
    as $\|x-y\|\to\infty$; see, e.g., \autocite[Theorem 1.6.5]{lawler2013intersections}. The boundary part has the same estimate, by the argument used in \eqref{eq:ss_massiveboundary}.
    
    For the massive part, if $0\le \lambda<2$, then both $G^m(x,y)$ and $G^m(z,x)$ are exponentially small by \Cref{thm:massiveanalysis} \ref{cond:subcrit}. Hence each derivative $\nabla^\alpha G^m$ remains exponentially small, for any multiindex $\alpha$. If $\lambda=2$, then by \Cref{thm:massiveanalysis} \ref{cond:crit}, each derivative $\nabla^\alpha G^m$ is bounded by $O(\|x-y\|^{-|\alpha|})$. The corresponding boundary term has the same estimate, by the same argument as in the massless case. Here, we use the symmetric form of the harmonic-measure decomposition \eqref{eq:GreenDecomp} with $x$ and $y$ interchanged, so that $H_N$ is taken from the fixed point $y$, and therefore does not depend on the variable $x$ on which the derivatives act.
    
    Since the interpolation layer has thickness of order $N$, we may choose a smooth cutoff interpolating between $0$ and $\bar{G}_N$ while preserving the derivative bounds \eqref{eq:interpolation}.
\end{proof}

For the fixed $y\in V_N^\delta$, the function $U\coloneq H_N(\cdot,y)-f_0$ is a solution of 
\begin{equation}\label{eq:competecomparison2}
    \begin{aligned}
        LU(x)&=-L f_0(x)=:g &x\in V_N,\\
        U(x)&=0 &x\notin V_N.
    \end{aligned}
\end{equation}
The zero exterior condition for $U$ follows directly from its definition. Since $G_N(\cdot,y)=0$ on $V_N^c$, we have
\[H_N(x,y)=\bar{G}_N(x,y)-G_N(x,y)=\bar{G}_N(x,y)\]
for $x\notin V_N$. By the definition of $f_0$, $f_0(x)=\bar{G}_N(x,y)$ outside $V_N^{\delta/2}$. Thus, for $x\notin V_N$, we have
\[U(x)=H_N(x,y)-f_0(x)=0.\]
Hence we may write $L_N U=LU$ on $V_N$.

Since $f_0 \equiv 0$ on $V_N^\delta$, it suffices to show that $\sup_x |U(x)| = O(1)$ rather than bounding $|\bar{G}_N - G_N|$ directly. Thus, it remains to analyze the function $U$. To this end, we study the function $g = -L f_0$.
\begin{lemma}\label[lemma]{lemma:sourcenorm}
    The function $g$ satisfies $\|g\|_{L^2(V_N)}=O(N^{-1})$.
\end{lemma}
\begin{proof}
    From \Cref{lemma:existence}, we can estimate $g(x)$ pointwise: On $V_N^{\delta/2}$,
    \begin{equation}\label{eq:source_ptwise_bdd}
        \begin{aligned}
            |g(x)|&=|-Lf_0(x)|=|(-\Delta+N^\lambda \Delta^2) f_0(x)|\\
            &\le |\Delta f_0(x)|+N^\lambda |\Delta^2 f_0(x)|\lesssim N^{-2}+N^{\lambda-4}.
        \end{aligned}
    \end{equation}
    On the other hand, for $x\in V_N\setminus V_N^{\delta/2}$, the definition of $f_0$ gives
    \[g(x)=-Lf_0(x)=-L\bar{G}_N(x,y)=-\delta_y(x)=0.\]
    Since $\lambda\le 2$, we obtain $|g(x)|\lesssim N^{-2}$, and hence
    \[\|g\|_{L^2(V_N)}^2 \lesssim\sum_{x\in V_N} (N^{-2})^2\asymp N^2 \times N^{-4}=N^{-2}.\qedhere\]
\end{proof}

When calculating the $L^2$ norm of the solution $U$, the boundary condition and its effects must be taken into account. For any functions $v,w:V_N\cup\partial_1V_N\to\mathbb{R}$ with $v(x)=w(x)=0$ on $x\in \partial_1V_N$, no boundary terms appear when applying summation by parts once. That is,
\begin{equation}\label{eq:SBP}
    \langle -\Delta_N v,w\rangle=\frac{1}{4d}\sum_{e}\langle \nabla_e v,\nabla_e w\rangle=:\frac{1}{2}\langle \nabla v,\nabla w\rangle=\langle v,-\Delta_N w\rangle.
\end{equation}
Here and below, $v$ and $w$ are understood as their zero extensions to $\mathbb{Z}^2$, and the inner products in the summation-by-parts identity are taken over $\mathbb{Z}^2$.

When applying bi-Laplacian, however, the zero boundary condition on $\partial_2 V_N$ induces the boundary error. In \autocite{kurt2009maximum}, its specific formula is as follows: for any $v$ and $w$ satisfying Dirichlet condition, the following relation
\begin{equation}\label{eq:posdef}
    \mathcal{D}(v,w)\coloneq \langle v,\Delta^2_N w\rangle =\langle \Delta_N v,\Delta_N w\rangle+\frac{1}{16}\sum_{x\in V_N}r(x)v(x)w(x)
\end{equation}
holds, where $r(x)\coloneq |\{y\in\mathbb{Z}^2\setminus V_N:\|x-y\|_1=1\}|$. Here the denominator $16$ comes from $\frac{1}{4d^2}$ with $d=2$. Hence, $r(x)\neq0$ if and only if $x\in\partial_1^- V_N$. Combining these observations, we have the following estimates.
\begin{lemma}\label[lemma]{lemma:solnorm}
    The function $U$ satisfies the following estimates
    \begin{equation}\label{eq:solnorm}
        \|U\|_{L^2(V_N)}=O(N),\ \|\nabla U\|_{L^2(V_N)}=O(1),\ \|\Delta U\|_{L^2(V_N)}=O(N^{-\lambda/2}).
    \end{equation}
\end{lemma}
\begin{proof}
    We use the Poincaré inequality: for a function $v$ on $V_N$ with the Dirichlet boundary condition,
    \begin{equation}\label{eq:poincare}
        \|v\|_{L^2(V_N)}\lesssim N\|\nabla v\|_{L^2(V_N)},
    \end{equation}
    where the coefficient $N$ arises from $\operatorname{diam}(V_N)$. Using \eqref{eq:SBP} and \eqref{eq:poincare}, 
    \begin{equation}\label{eq:solgradestimate}
        \begin{aligned}
            \frac{1}{2}\|\nabla U\|_{L^2(V_N)}^2 &\overset{\mathclap{\eqref{eq:SBP}}}{\le}\langle-\Delta_N U,U\rangle_{V_N} \overset{\mathclap{\eqref{eq:posdef}}}{\le} \langle-\Delta_N U,U\rangle_{V_N}+N^\lambda \langle \Delta^2_N U,U\rangle_{V_N}\\
            &=\langle LU,U\rangle_{V_N}\le \|LU\|_{L^2(V_N)}\|U\|_{L^2(V_N)}\\
            &\overset{\mathclap{\eqref{eq:poincare}}}{\lesssim} N\|LU\|_{L^2(V_N)}\|\nabla U\|_{L^2(V_N)}.
        \end{aligned}
    \end{equation}
    \Cref{lemma:sourcenorm} with \eqref{eq:solgradestimate} yields $\|\nabla U\|_{L^2(V_N)}=O(1)$. Using \eqref{eq:poincare} again, we have $\|U\|_{L^2(V_N)}\lesssim N\|\nabla U\|_{L^2(V_N)}=O(N)$.
    
    The function $\Delta_N U$ no longer satisfies the boundary condition. Instead of dealing with the boundary,  we use \eqref{eq:posdef}.
    \begin{align*}
        \|\Delta_N U\|_{L^2(V_N)}^2&=\langle \Delta_N U, \Delta_N U\rangle \le \langle \Delta_N U, \Delta_N U\rangle +\frac{1}{16}\sum_{x\in V_N} r(x)U(x)^2\\
        &\overset{\mathclap{\eqref{eq:posdef}}}{=}\langle U, \Delta^2_N U\rangle \le \frac{1}{2}N^{-\lambda}\|\nabla U\|_{L^2(V_N)}^2 +\langle U, \Delta^2_N U\rangle\\
        &\overset{\mathclap{\eqref{eq:SBP}}}{\le} N^{-\lambda}\langle U, -\Delta_N U\rangle +\langle U, \Delta^2_N U\rangle\\
        &=N^{-\lambda} \langle U, LU\rangle\\
        &\le N^{-\lambda}\|U\|_{L^2(V_N)}\|LU\|_{L^2(V_N)}.
    \end{align*}
    Combining this with \Cref{lemma:sourcenorm} and the previous estimate for $U$, we obtain the desired norm estimate $\|\Delta U\|_{L^2(V_N)}=O(N^{-\lambda/2})$.
\end{proof}

\subsubsection{Fractional Sobolev Embedding}
Now we discuss the fractional Sobolev embedding and some preliminaries. First, we define the DST coefficient for a function $f:V_N\to\mathbb{R}$ as 
\[\widehat{f}(k)\coloneq\langle f,\varphi_k\rangle,\]
where $\varphi_k$ is defined in \eqref{eq:eigen}. Then, for $s\ge0$, we define the fractional Laplacian $(-\Delta_N)^s f$ as
\[(-\Delta_N)^s f(x)\coloneq\sum_k\widehat{D}(k)^s\widehat{f}(k) \varphi_k(x)\]
and the fractional Sobolev norm $\|\cdot\|_{H^s(V_N)}$ as
\[\|f\|_{H^s(V_N)}^2\coloneq \sum_k (1+N^2\widehat{D}(k))^s|\widehat{f}(k)|^2.\]
Here $\widehat{D}(k)$ is the eigenvalue of $-\Delta_N$, defined in \eqref{eq:eigen}. This fractional Sobolev norm gives the Sobolev embedding theorem.
\begin{lemma}\label[lemma]{lemma:fracSobolev}
    Let $s>d/2$. For a function $f:V_N\to\mathbb{R}$, 
    \[\sup_{x\in V_N}|f(x)|\lesssim N^{-d/2}\|f\|_{H^s(V_N)}.\]
    In particular, if $\|f\|_{H^s(V_N)}\lesssim N^{d/2}$, then $\sup_{x\in V_N}|f(x)|\le C$ for some $C>0$, which is independent of $N$.
\end{lemma}
\begin{proof}
    Cauchy-Schwarz inequality gives
    \begin{align*}
        \sum_k |\widehat{f}(k)|&\le\left(\sum_k (1+N^2\widehat{D}(k))^{-s}\right)^{1/2}\left(\sum_k (1+N^2\widehat{D}(k))^s |\widehat{f}(k)|^2\right)^{1/2}\\
        &=\left(\sum_k (1+N^2\widehat{D}(k))^{-s}\right)^{1/2}\|f\|_{H^s(V_N)}.
    \end{align*}
    Since we have $N^2\widehat{D}(k)\gtrsim \|k\|^2$, 
    \[\sum_k (1+N^2\widehat{D}(k))^{-s}\lesssim \sum_k (1+\|k\|^2)^{-s}<\infty,\]
    since $s>d/2$. Hence, $\sum_k |\widehat{f}(k)|\lesssim \|f\|_{H^s(V_N)}$. From these,
    \[|f(x)|=\left|\sum_k \widehat{f}(k)\varphi_k(x)\right|\lesssim N^{-d/2}\sum_k |\widehat{f}(k)|\lesssim N^{-d/2}\|f\|_{H^s(V_N)}.\]
    The coefficient $N^{-d/2}$ comes from the normalization constant $(N+1)^{-d/2}$ of $\varphi_k(x)$.
\end{proof}
\begin{remark}
    Instead of using the fractional Sobolev norm, we can define an integer-order Sobolev norm $H^k$ by the finite difference operators. This norm also yields a Sobolev embedding theorem. However, \Cref{lemma:solnorm} only yields $\|U\|^2_{H^2(V_N)}\asymp N^{4-\lambda}$, which is insufficient to deduce the uniform $O(1)$ bound unless $\lambda=2$. This is why we instead work with a fractional Sobolev norm of order $1+\varepsilon<2$.
\end{remark}

Although the fractional Laplacian $(-\Delta_N)^s$ is well defined for any $s\ge0$, our analysis involves $\Delta_N^2$ and related operators. As already mentioned, we have $(-\Delta_N)^2\neq \Delta^2_N$, and hence the equation for $U$ cannot be directly rewritten in terms of the spectral powers of $-\Delta_N$. To avoid this issue, we consider a smooth cutoff function $\eta$ satisfying
\[0\le\eta\le1,\ \eta\equiv 1\text{ on } V_N^\delta,\ \eta\equiv 0 \text{ on } V_N\setminus V_N^{(\delta+\delta')/2},\ \sup|\nabla^\alpha \eta|\lesssim N^{-|\alpha|}.\]
Here $\delta'$ is the parameter used in the definition of $f_0$. This choice of $\eta$ ensures that $\partial_2\operatorname{supp}\nabla\eta\subset V_N^{\delta'}$, since $\delta/2<\delta'<\delta$. We abbreviate $\tilde{U}=U\eta$ and introduce a new source term $\tilde{g}=L\tilde{U}$. (Recall that $U=H_N(\cdot,y)-f_0$; see \eqref{eq:competecomparison2}.) Then $\tilde{U}$ and $\tilde{g}$ preserve the norm scaling of the initial functions $U$ and $g$. Namely, the following lemma holds.
\begin{lemma}\label[lemma]{lemma:cutoffsolnorm}
    The function $\tilde{U}$ satisfies 
    \[\|\tilde{U}\|_{L^2(V_N)}=O(N),\ \|\nabla \tilde{U}\|_{L^2(V_N)}=O(1),\ \|\Delta \tilde{U}\|_{L^2(V_N)}=O(N^{-\lambda/2}).\]
\end{lemma}
\begin{proof}
    By \Cref{lemma:solnorm} and the assumption of $\eta$, the bound for $\tilde{U}$ follows directly. Namely,
    \[\|\tilde{U}\|^2_{L^2(V_N)}=\sum_{x\in V_N}\tilde{U}(x)^2=\sum_{x\in V_N}(U(x)\eta(x))^2\le \sum_{x\in V_N}U(x)^2=\|U\|^2_{L^2(V_N)}.\]
    When calculating $\|\nabla\tilde{U}\|$, we use the discrete Leibniz rule 
    \[\nabla_e (f_1 f_2)(x)=\nabla_e f_1(x) f_2(x)+f_1(x+e)\nabla_e f_2(x),\]
    and hence $|\nabla_e\tilde{U}(x)|\lesssim |\nabla_e U(x)|+\frac{1}{N}|U(x+e)|$. Combining this with \Cref{lemma:solnorm}, 
    \begin{align*}
        \|\nabla\tilde{U}\|^2_{L^2(V_N)}&=\frac{1}{4} \sum_e \sum_x (\nabla_e \tilde{U}(x))^2\\
        &\lesssim \frac{1}{4} \sum_e \sum_x \left[(\nabla_e U(x))^2+\frac{1}{N^2}|U(x+e)|^2\right]\\
        &=O(1).
    \end{align*}

    When estimating $\|\Delta\tilde{U}\|$, we use the commutator identity
    \[\Delta(U\eta)=\eta \Delta U +[\Delta,\eta] U,\]
    where $[A,B]\coloneq AB-BA$ is the commutator operator. In this context, we regard $\eta:v\mapsto \eta v$ as an operator. Since $\sup \eta=1$, the norm of the first term is bounded by $\|\Delta U\|_{L^2(V_N)}=O(N^{-\lambda/2})$. For the second term, using the commutator operator identity
    \begin{equation}\label{eq:comidentity}
        [\Delta,f_1]f_2(x)=f_2(x)\Delta f_1(x)+\nabla f_1(x) \cdot \nabla f_2 (x),
    \end{equation}
    the commutator part can be estimated as
    \begin{align*}
        \|[-\Delta,\eta]U\|_{L^2(V_N)}^2&\overset{\mathclap{\eqref{eq:comidentity}}}{\lesssim} \sum_{x\in V_N} U(x)^2 (-\Delta \eta(x))^2 +\sum_{x\in V_N} \nabla \eta(x)^2 \nabla U(x)^2\\
        &\lesssim N^{-4}\|U\|_{L^2(V_N)}^2+N^{-2}\|\nabla U\|_{L^2(V_N)}^2 \overset{\mathclap{\eqref{eq:solnorm}}}{\lesssim} N^{-2}.
    \end{align*}
    Combining these, we have $\|\Delta\tilde{U}\|_{L^2(V_N)}=O(N^{-\lambda/2}) $.
\end{proof}
Using the commutator identity again, we can estimate the norm of $\tilde{g}$.
\begin{lemma}\label[lemma]{lemma:cutoffsourcenorm}
    We have $\|\tilde{g}\|_{L^2(V_N)}=O(N^{-1})$.
\end{lemma}
\begin{proof}
    We have $\tilde{g}=L(U\eta)=\eta LU+[L,\eta] U$, and 
    \[\|\eta LU\|_{L^2(V_N)}\le \|\eta\|_{L^\infty(V_N)}\|LU\|_{L^2(V_N)}=O(N^{-1})\]
    directly follows. It remains to estimate the commutator term
    \[[L,\eta]U=[-\Delta_N,\eta]U+N^\lambda [\Delta_N^2,\eta]U.\]

    The Laplacian part $[-\Delta_N,\eta]U=(-\Delta_N \eta)U-\nabla \eta \cdot \nabla U$ has the following upper bound:
    \begin{equation}\label{eq:laplace_commutator}
        \begin{aligned}
            \|(-\Delta_N\eta)U\|_{L^2(V_N)}&\lesssim N^{-2}\|U\|_{L^2(V_N)}=O(N^{-1}),\\
            \|\nabla \eta \cdot \nabla U\|_{L^2(V_N)}&\lesssim N^{-1}\|\nabla U\|_{L^2(V_N)}=O(N^{-1}).
        \end{aligned}
    \end{equation}
    When calculating the bi-Laplacian part $[\Delta^2_N,\eta]U$, since $\eta$ vanishes near the boundary, we may identify $(\Delta_N)^2$ with $\Delta^2_N$ on the support. Hence, suppressing shifts in the indices for notational convenience,
    \begin{equation}\label{eq:bilaplace_commutator}
        \begin{aligned}
            [\Delta^2,\eta]U&=\Delta[\Delta,\eta]U+[\Delta,\eta]\Delta U\\
            &=\Delta (U\Delta\eta +\nabla \eta\cdot\nabla U) +\Delta U \Delta \eta + \nabla \eta \cdot \nabla \Delta U\\
            &=2\Delta U\Delta \eta +U \Delta^2 \eta + 2\nabla \Delta \eta \cdot \nabla U +2\nabla \eta\cdot \nabla \Delta U + \nabla^2 \eta :\nabla^2 U,
        \end{aligned}
    \end{equation}
    where $\nabla^2 \eta:\nabla^2 U(x)\coloneq \frac{1}{16}\sum_{e,e'} \nabla_e\nabla_{e'} \eta (x)\nabla_e\nabla_{e'} U(x)$, and the sum is taken over all directed edges connected to $x$. By \Cref{lemma:solnorm}, all terms in \eqref{eq:bilaplace_commutator} except $\nabla \eta\cdot\nabla\Delta U$ satisfy the estimates
    \begin{equation}\label{eq:bilaplace_commutator_estimate}
        \begin{aligned}
            \|\Delta U \Delta \eta\|_{L^2(V_N)}&\lesssim N^{-2}\|\Delta U\|_{L^2(V_N)}=O(N^{-2-\lambda/2}),\\
            \|\nabla^2 \eta :\nabla^2 U\|_{L^2(V_N)}&\lesssim N^{-2}\|\nabla^2 U\|_{L^2(V_N)}\\
            &\lesssim N^{-2}\|\Delta U\|_{L^2(V_N)}=O(N^{-2-\lambda/2}),\\
            \|U \Delta^2 \eta\|_{L^2(V_N)}&\lesssim N^{-4}\|U\|_{L^2(V_N)}=O(N^{-3}),\\
            \|\nabla \Delta \eta\cdot\nabla U\|_{L^2(V_N)}&\lesssim N^{-3}\|\nabla U\|_{L^2(V_N)}=O(N^{-3}).
        \end{aligned}
    \end{equation}
    After multiplying by the factor $N^\lambda$, all of them are of order $O(N^{-1})$.
    
    To estimate $\nabla \eta\cdot\nabla\Delta U$, we use a Caccioppoli-type estimate. Consider a new smooth cutoff $\chi$, which satisfies $0\le \chi\le 1$, $\operatorname{supp}\chi \subset V_N^{\delta'}\setminus\partial^-_2V_N^{\delta'}$, and $(\operatorname{supp}\eta)\cup\partial_2 (\operatorname{supp}\eta) \subset \{\chi=1\}$. Here $\delta'$ is that of \Cref{lemma:existence}. The support condition is to guarantee $\operatorname{supp} \chi\subset \{g=0\}\setminus \partial^-\{g=0\}$ and $\operatorname{supp}\nabla\eta\subset\{\chi=1\}\setminus\partial^-\{\chi=1\}$.

    From these conditions, we may regard $\Delta_N$ and $\Delta^2_N$ as the whole space operators $\Delta$ and $\Delta^2$, respectively. Let $A$ be the support of $\nabla_{e_j} \eta$, for a suitable direction $e_j$. From the construction of $\chi$, we have $A\cup \partial A\subset\{\chi=1\}$. On $\operatorname{supp}\chi$, we have $-\Delta U +N^\lambda (-\Delta)^2 U=0$, and hence it holds on $A$. We set $w=-\Delta U$. Then we have
    \begin{equation}\label{eq:Caccioppoli}
        \langle \chi^2 w, w-N^\lambda \Delta w\rangle= \|\chi w\|^2 +N^\lambda \langle \chi^2 w, -\Delta w\rangle=0.
    \end{equation}
    Here the inner product is on $V_N$ due to the support condition of $\chi$. Applying summation by parts and discrete Leibniz rule, from \eqref{eq:Caccioppoli},
    \begin{equation}\label{eq:Caccioppoli_Error}
        \begin{aligned}
            \langle \chi^2 w, -\Delta w\rangle&=\frac{1}{2}\langle \nabla (\chi^2 w), \nabla w\rangle\\
            &=\frac{1}{8}\sum_{x,e} \nabla_e w(x)[\nabla_e w(x)\chi^2(x+e) +w(x) \nabla_e \chi^2(x)]\\
            &=\frac{1}{8}\sum_e \|\chi(\cdot+e)\nabla_e w \|_{L^2(V_N)}^2 \\
            &\phantom{=}+\frac{1}{8}\sum_{x,e} w(x)\nabla_e w(x) (\chi(x)+\chi(x+e))\nabla_e \chi(x).
        \end{aligned}
    \end{equation}
    Using Young's inequality, for any $\varepsilon>0$,
    \begin{equation}\label{eq:Caccioppoli_Young}
        \begin{aligned}
            &\left|\sum_{x,e} w(x)\nabla_e w(x) (\chi(x)+\chi(x+e))\nabla_e \chi(x)\right|\\
            \le& 2\varepsilon \sum_{x,e} (\nabla_e w(x))^2 (\chi^2(x)+\chi^2(x+e))\\
            &+\frac{1}{4\varepsilon}\sum_{x,e} (w(x) \nabla_e \chi(x))^2.
        \end{aligned}
    \end{equation}
    Since we take the sum over all $x\in V_N$ and all edges $e$, if we extend by $0$ outside $V_N$, we have
    \[\sum_{x,e}(\nabla_e w(x) \chi(x))^2=\sum_{x,e}(\nabla_e w(x) \chi(x+e))^2.\]
    For sufficiently small $\varepsilon$, this term can be absorbed into $\sum_e \|\chi(\cdot+e) \nabla_e w\|_{L^2(V_N)}^2$. For such $\varepsilon$, the second term is simply controlled by $N^{-2}\|w\|_{L^2(\operatorname{supp}\nabla\chi)}^2\le N^{-2}\|w\|_{L^2(V_N)}^2$. In summary, 
    \begin{align*}
        0&=\|\chi w\|_{L^2(V_N)}^2 +N^\lambda \langle \chi^2 w, -\Delta w\rangle_{L^2(V_N)} \\
        &\ge \|\chi w\|_{L^2(V_N)}^2 +cN^\lambda \sum_e \|\chi(\cdot+e) \nabla_e w\|_{L^2(V_N)}^2 -CN^{\lambda-2} \|w\|_{L^2(V_N)}^2 \\
        &\overset{\eqref{eq:solnorm}}{\ge} cN^\lambda \sum_e \|\chi(\cdot+e) \nabla_e w\|_{L^2(V_N)}^2 -CN^{-2}.
    \end{align*}
    This implies that
    \[N^{-1-\lambda/2}\gtrsim \left(\sum_e \|\chi(\cdot+e) \nabla_e w\|_{L^2(V_N)}^2\right)^{1/2}\gtrsim \|\nabla w\|_{L^2(A)}= \|\nabla\Delta U\|_{L^2(A)},\]
    since $A\cup\partial A\subset\{\chi=1\}$. Hence we conclude that, by repeating this argument for each direction,
    \begin{equation}\label{eq:bilaplace_exceptional}
        \|\nabla\eta\cdot\nabla\Delta U\|\lesssim N^{-1}\|\nabla\Delta U\|_{L^2(A)}\lesssim N^{-2-\lambda/2}.
    \end{equation}

    Combining \eqref{eq:laplace_commutator}, \eqref{eq:bilaplace_commutator_estimate}, and \eqref{eq:bilaplace_exceptional}, $\|\tilde{g}\|_{L^2(V_N)}=O(N^{-1})$ is shown.
\end{proof}

Now we can estimate the $H^{1+\varepsilon}$ norm of $\tilde{U}$, using \Cref{lemma:cutoffsolnorm} and \Cref{lemma:cutoffsourcenorm}.
\begin{proposition}\label[proposition]{thm:fractionalnorm}
    Let $\tilde{U}$ be defined as above. Then for any $\varepsilon\in(0,1)$,
    \[\|\tilde{U}\|_{H^{1+\varepsilon}(V_N)}\lesssim N.\]
    As a consequence, $\sup_{x\in V_N}|\tilde{U}(x)|=O(1)$.
\end{proposition}
\begin{proof}
    Using the inequality $1\lesssim N^2\widehat{D}(k)$, we obtain
    \begin{equation}\label{eq:1+epsilon_norm}
        \begin{aligned}
            \|\tilde{U}\|_{H^{1+\varepsilon}(V_N)}^2&=\sum_k (1+N^2\widehat{D}(k))^{1+\varepsilon}|\widehat{\tilde{U}}(k)|^2\\
            &\lesssim \sum_k N^{2(1+\varepsilon)}\widehat{D}(k)^{1+\varepsilon}|\widehat{\tilde{U}}(k)|^2\\
            &= N^{2(1+\varepsilon)}\|(-\Delta)^{(1+\varepsilon)/2}\tilde{U}\|_{L^2(V_N)}^2.
        \end{aligned}
    \end{equation}
    Recall that $\widehat{\tilde{U}}(k)=\langle\tilde{U},\varphi_k\rangle$ is the DST coefficient for $\tilde{U}$.
    
    By the Plancherel theorem, we have
    \begin{equation}\label{eq:plancherel}
        \begin{aligned}
            \|(-\Delta)^{(1+\varepsilon)/2}\tilde{U}\|_{L^2(V_N)}^2&= \sum \widehat D(k)^{1+\varepsilon} |\widehat{\tilde{U}}(k)|^2\\
            &=\langle -\Delta \tilde{U}, (-\Delta_N)^\varepsilon \tilde{U}\rangle\ge0.
        \end{aligned}
    \end{equation}
    Similarly,
    \begin{equation}\label{eq:plancherel2}
        \begin{aligned}
            \|(-\Delta_N)^{1+\varepsilon/2}\tilde{U}\|_{L^2(V_N)}^2&= \sum \widehat D(k)^{2+\varepsilon} |\widehat{\tilde{U}}(k)|^2\\
            &=\langle (-\Delta_N)^2 \tilde{U}, (-\Delta_N)^\varepsilon \tilde{U}\rangle\ge0.
        \end{aligned}
    \end{equation}

    Using $1\lesssim N^2\widehat{D}(k)$ again, it follows that
    \begin{equation}\label{eq:fractional_norm}
        \begin{aligned}
            \|(-\Delta_N)^{(1+\varepsilon)/2}\tilde{U}\|_{L^2(V_N)}^2&= \sum \widehat{D}(k)^{1+\varepsilon}|\widehat{\tilde{U}}(k)|^2\\
            &=\sum \widehat{D}(k)^{1-\varepsilon}\widehat{D}(k)^{2\varepsilon}|\widehat{\tilde{U}}(k)|^2\\
            &\gtrsim N^{-2(1-\varepsilon)}\sum \widehat{D}(k)^{2\varepsilon}|\widehat{\tilde{U}}(k)|^2\\
            &=N^{-2(1-\varepsilon)} \|(-\Delta_N)^\varepsilon \tilde{U}\|_{L^2(V_N)}^2.
        \end{aligned}
    \end{equation}
    
    Therefore,
    \begin{equation}\label{eq:(1+epsilon)/2_norm}
        \begin{aligned}
            \|(-\Delta_N)^{(1+\varepsilon)/2}\tilde{U}\|_{L^2(V_N)}^2&\overset{\mathclap{\eqref{eq:plancherel}}}{=}\langle-\Delta_N \tilde{U}, (-\Delta_N)^\varepsilon \tilde{U}\rangle\\
            &\overset{\mathclap{\eqref{eq:plancherel2}}}{\le} \langle-\Delta_N \tilde{U}, (-\Delta_N)^\varepsilon \tilde{U}\rangle+N^\lambda \langle(-\Delta_N)^2 \tilde{U}, (-\Delta_N)^\varepsilon \tilde{U}\rangle\\
            &=\langle \tilde{g},(-\Delta_N)^\varepsilon\tilde{U}\rangle.
        \end{aligned}
    \end{equation}

    From \Cref{lemma:cutoffsourcenorm},
    \begin{equation}\label{eq:epsilon_norm}
        \begin{aligned}
            \langle \tilde{g},(-\Delta_N)^\varepsilon\tilde{U}\rangle&\le  \|\tilde{g}\|_{L^2(V_N)}\|(-\Delta_N)^\varepsilon \tilde{U}\|_{L^2(V_N)}\\
            &\overset{\mathclap{\eqref{eq:fractional_norm}}}{\lesssim} N^{1-\varepsilon}\|\tilde{g}\|_{L^2(V_N)} \|(-\Delta_N)^{(1+\varepsilon)/2}\tilde{U}\|_{L^2(V_N)}.\\
            &\lesssim N^{-\varepsilon}\|(-\Delta_N)^{(1+\varepsilon)/2}\tilde{U}\|_{L^2(V_N)}.
        \end{aligned}
    \end{equation}

    \eqref{eq:(1+epsilon)/2_norm} and \eqref{eq:epsilon_norm} conclude that $\|(-\Delta_N)^{(1+\varepsilon)/2}\tilde{U}\|_{L^2(V_N)}\lesssim N^{-\varepsilon}$. In summary, 
    \[
        \|\tilde{U}\|_{H^{1+\varepsilon}(V_N)}^2\overset{\mathclap{\eqref{eq:1+epsilon_norm}}}{\lesssim} N^{2(1+\varepsilon)}\|(-\Delta_N)^{(1+\varepsilon)/2}\tilde{U}\|_{L^2(V_N)}^2\lesssim N^2.
    \]
    This completes the proof.
\end{proof}
\Cref{lemma:fracSobolev} and \Cref{thm:fractionalnorm} concludes the proof of \Cref{thm:competecomparison}. We now complete the proof of \Cref{thm:compregime}.
\begin{proof}[Proof of \Cref{thm:compregime}]
    From \Cref{thm:competecomparison}, we obtain 
    \[\sup_{x\in V_N}|\tilde{U}(x)|=\sup_{x\in V_N}|\eta(x)U(x)|=O(1).\]
    If $x$ lies in $V_N^\delta$, the cutoff function $\eta(x)$ is identically $1$. Thus for each $y\in V_N^\delta$, we have $|U(x)|=|G_N(x,y)-\bar{G}_N(x,y)|$, and it is bounded by some constant $C$, which depends only on $\delta$ and $\lambda$. From \Cref{thm:sslongrange} and \Cref{thm:ssshortrange}, we can estimate the asymptotic behavior of $\bar{G}_N(x,y)$ with error at most $O(1)$, and the upper bound $C$ can be absorbed in this $O(1)$ error term.
\end{proof}
    % !TEX root = ../main.tex
\section{Proof of \texorpdfstring{\Cref{thm:domregime}}{Theorem 1.2}}
In this section, we prove \Cref{thm:domregime}, without appealing to \eqref{eq:partialfrac}. The main difference between \Cref{section:proofcompregime} and this section is that $G_N(\cdot,y)$ is regarded as a real function in \Cref{section:proofcompregime}, whereas we regard $G_N(\cdot,y)$ as a matrix, and we use some matrix identities. 

\subsection{DGFF Dominant Regime}\label{ch:DGFFdom}
We first consider the case $\lambda<0$. The goal is to show that the covariance of the two-dimensional DGFF $G^0_N=(-\Delta_N)^{-1}$ approximates $G_N$ well. However, a direct comparison between $G_N$ and $G^0_N$ is technically more involved. Instead of direct comparison with $G^0_N$, we introduce an auxiliary approximation by $\bar{G}_N$.

We summarize the goal of this section in the following proposition.
\begin{proposition}\label{thm:AuxApprox}
    Fix $\lambda<0$, $d=2$, and let $G_N$ and $\bar{G}_N$ be as in \eqref{eq:cl} and \eqref{eq:ss}, respectively. For a fixed $\delta\in(0,1)$ and $x,y\in V_N^\delta$, we have
    \begin{equation}\label{eq:AuxApprox2}
        |\bar{G}_N(x,y)-G^0_N(x,y)|=O(N^\lambda)
    \end{equation}
    and
    \begin{equation}\label{eq:AuxApprox1}
        |G_N(x,y)-\bar{G}_N(x,y)|=O(N^{\lambda-1})
    \end{equation}
    uniformly.
\end{proposition}

\subsubsection{Auxiliary Approximation}
We first prove the auxiliary estimate \eqref{eq:AuxApprox2}. As already discussed in \Cref{ch:strategy}, $\bar{G}_N$ is diagonalized by the eigenbasis $\{\varphi_k\}_k$ given in \eqref{eq:eigen}, with the corresponding eigenvalue $(\widehat{D}(k)+N^\lambda\widehat{D}(k)^2)^{-1}$. The eigenvalue comparison yields that
\begin{equation}
    0\le \frac{1}{\widehat{D}(k)}-\frac{1}{\widehat{D}(k)+N^\lambda\widehat{D}(k)^2}=\frac{N^\lambda}{1+N^\lambda \widehat{D}(k)}\le N^\lambda.
\end{equation}
Since all eigenvalues of $G^0_N-\bar{G}_N$ are nonnegative, $G^0_N-\bar{G}_N$ is positive semidefinite. Furthermore, the matrix norm equals the maximal eigenvalue, from the self-adjoint property. That is, we have
\begin{equation}
    \|G^0_N-\bar{G}_N\|_{L^2(V_N)\to L^2(V_N)}\le N^\lambda.
\end{equation}
Therefore we deduce that
\begin{equation}\label{eq:AuxApproxUniform}
    \begin{aligned}
        |G^0_N(x,y)-\bar{G}_N(x,y)|&=|\langle \delta_x,(G^0_N-\bar{G}_N)\delta_y\rangle|\\
        &\le \|G^0_N-\bar{G}_N\|_{L^2(V_N)\to L^2(V_N)}\le N^\lambda.
    \end{aligned}
\end{equation}
This concludes the estimate \eqref{eq:AuxApprox2}. In fact, this estimate holds uniformly for all $x,y\in V_N$.

\subsubsection{Main Approximation}
We next prove the main estimate \eqref{eq:AuxApprox1}. Let $\bar{L}_N\coloneq -\Delta_N+N^\lambda (\Delta_N)^2$, and $D_N\coloneq  \Delta_N^2-(\Delta_N)^2=N^{-\lambda} (L_N-\bar{L}_N)$. By \eqref{eq:posdef}, for any function $f:\mathbb{Z}^2\to \mathbb{R}$ zero outside $V_N$,
\[\langle f,D_N f\rangle =\frac{1}{16}\sum_{x\in V_N}r(x)f(x)^2\ge0.\]
Hence, $D_N$ is positive semidefinite. Both $L_N$ and $\bar{L}_N$ are symmetric, and thus we conclude that $\bar{G}_N-G_N$ is positive semidefinite, by the standard Löewner order argument. Testing the positive semidefinite matrix $\bar{G}_N-G_N$ against the standard coordinate vectors $\delta_x\pm\delta_y$, we obtain the following inequality: 
\begin{equation}\label{eq:covineq1}
    \begin{aligned}
        &\langle (\delta_x\pm \delta_y),(\bar{G}_N-G_N)(\delta_x\pm\delta_y)\rangle\\
        =&(\bar{G}_N-G_N)(x,x)\pm 2(\bar{G}_N-G_N)(x,y)+(\bar{G}_N-G_N)(y,y)\ge0.
    \end{aligned}
\end{equation}
for $x,y\in V_N^\delta$. Here $\delta_x$ denotes the Kronecker delta function $\delta_x(z)=\mathbbm{1}\{x=z\}$, viewed as the standard coordinate vector in $L^2(V_N)$.

Consequently,
\begin{equation}\label{eq:covineq2}
    |(\bar{G}_N-G_N)(x,y)|\lesssim (\bar{G}_N-G_N)(x,x)+(\bar{G}_N-G_N)(y,y).
\end{equation}
This yields that it suffices to estimate the diagonal term $|\bar{G}_N-G_N|(x,x)$. Hence, the goal of this section is to find a uniform estimate of $|\bar{G}_N-G_N|(x,x)$, for $x\in V_N^\delta$. Throughout this section, we fix $x,y \in V_N^\delta$.

The resolvent identity yields
\begin{equation}\label{eq:resolvent}
    \begin{aligned}
        \bar{G}_N-G_N&=\bar{G}_N(L_N-\bar{L}_N)G_N=N^\lambda \bar{G}_N D_N G_N\\
        &= G_N(L_N-\bar{L}_N)\bar{G}_N=N^\lambda G_N D_N \bar{G}_N.
    \end{aligned}
\end{equation}
Testing against the coordinate vector $\delta_x$, it follows that
\begin{equation}\label{eq:varerr}
    \begin{aligned}
        |\bar{G}_N(x,x)-G_N(x,x)|&=|\langle\delta_x,(\bar{G}_N-G_N)\delta_x \rangle|\\
        &=N^\lambda |\langle \sqrt{D_N}\bar{G}_N \delta_x, \sqrt{D_N} G_N \delta_x\rangle|\\
        &\le N^\lambda \|\sqrt{D_N}\bar{G}_N\delta_x\|\|\sqrt{D_N}G_N\delta_x\|,
    \end{aligned}
\end{equation}
where the matrix $\sqrt{D_N}$ is the unique symmetric, positive semidefinite matrix $M_N$ satisfying $M_N^2=D_N$.

In \eqref{eq:varerr}, we can represent the norm as the following inner product form
\[\|\sqrt{D_N}G_N\delta_x\|^2=\langle \sqrt{D_N}G_N\delta_x, \sqrt{D_N}G_N\delta_x\rangle=\langle\delta_x, G_N D_N G_N\delta_x\rangle,\]
and similarly for $\bar{G}_N$. 

From $D_N=\Delta_N^2-(\Delta_N)^2=N^{-\lambda}(L_N-\bar{L}_N)$, we have
\begin{align*}
    N^\lambda G_N D_N G_N&=G_N-G_N \bar{L}_N G_N,\\
    N^\lambda \bar{G}_N D_N \bar{G}_N&=\bar{G}_N L_N \bar{G}_N-\bar{G}_N.
\end{align*}
Their difference is 
\[N^\lambda (\bar{G}_N D_N \bar{G}_N-G_N D_N G_N)=\bar{G}_N L_N \bar{G}_N+G_N \bar{L}_N G_N-(\bar{G}_N+G_N).\]
A direct algebraic simplification gives
\begin{equation}\label{eq:resolvent_decomposition}
    \bar{G}_N L_N \bar{G}_N+G_N \bar{L}_N G_N-(\bar{G}_N+G_N)=(\bar{G}_N-G_N)(L_N+\bar{L}_N)(\bar{G}_N-G_N),
\end{equation}
where we used the identity $AQA-BQB=(A-B)QA+BQ(A-B)$ for square matrices $A$, $B$ and $Q$, together with the inverse relation between $G_N$ and $L_N$, and between $\bar{G}_N$ and $\bar{L}_N$.

All matrices in \eqref{eq:resolvent_decomposition} are symmetric and positive semidefinite, and therefore $\bar{G}_N D_N \bar{G}_N-G_N D_N G_N$ is symmetric and positive semidefinite. Hence, the right side of \eqref{eq:varerr} is bounded by 
\begin{equation}\label{eq:bound}
    N^\lambda\|\sqrt{D_N}\bar{G}_N\delta_x\|^2=N^\lambda \langle \delta_x, \bar{G}_N D_N \bar{G}_N \delta_x\rangle.
\end{equation}

Now we consider the integral kernel form of \eqref{eq:bound} componentwise. For convenience, we introduce some operators. Let $R_N$ be the domain restriction operator from $\mathbb{Z}^2$ to $V_N$, and $R_N^*$ is its adjoint, the zero extension embedding operator from $V_N$ to $\mathbb{Z}^2$. Then the operators $\Delta_N$ and $\Delta_N^2$ can be expressed as 
\begin{equation}
    \begin{aligned}
        \Delta_N&=R_N \Delta R_N^*,\\
        \Delta_N^2&=R_N \Delta^2 R_N^*.
    \end{aligned}
\end{equation} 
From these expressions, we have
\begin{equation}\label{eq:DiffDecomp}
    D_N=\Delta^2_N-(\Delta_N)^2=R_N\Delta (I_{\mathbb{Z}^2}-R_N^* R_N)\Delta R_N^*.
\end{equation}
The operator $(I_{\mathbb{Z}^2}-R_N^* R_N)$ in \eqref{eq:DiffDecomp} is the orthogonal projection operator from $\mathbb{Z}^2$ to $V_N^c$. We abbreviate this as $P_N$. In particular, $P_N^2=P_N$, and \eqref{eq:DiffDecomp} can be written as
\begin{equation}\label{eq:DiffDecomp2}
    D_N=(P_N\Delta R_N^*)^*P_N\Delta R_N^*.
\end{equation}

The kernel forms of the operators $R_N^*$ and $P_N$ are
\begin{equation}
    \begin{aligned}
        R_N^*(z,w)&=\delta_{z}(w),&& \forall z\in\mathbb{Z}^2, w\in V_N,\\
        P_N(z,w)&=\delta_{z}(w)\mathbbm{1}_{z\notin V_N},&& \forall z,w\in\mathbb{Z}^2.
    \end{aligned}
\end{equation}
Then we have the following identity
\begin{equation}\label{eq:kernelrep}
    \begin{aligned}
        P_N\Delta R_N^*(z,w)&=\sum_{u,v\in\mathbb{Z}^2}P_N(z,v)\Delta(v,u)R_N^*(u,w)\\
        &=\Delta(z,w)\mathbbm{1}_{z\notin V_N}.
    \end{aligned}
\end{equation}
We put \eqref{eq:kernelrep} into \eqref{eq:DiffDecomp2}.
\begin{equation}\label{eq:DiffConv}
    \begin{aligned}
        D_N(z,w)&=\sum_{v\in\mathbb{Z}^2} (P_N\Delta R_N^*) (v,z) (P_N\Delta R_N^*)(v,w)\\
        &=\sum_{v\in\mathbb{Z}^2} \Delta(v,z)\mathbbm{1}_{v\notin V_N} \Delta(v,w)\mathbbm{1}_{v\notin V_N}\\
        &=\sum_{v\notin V_N;\|v-z\|=1} \Delta(v,z)\Delta(v,w).
    \end{aligned}
\end{equation}
The kernel of two-dimensional discrete Laplacian $\Delta$ is given by
\[\Delta(z,w)=\begin{cases}
    -1&z=w,\\
    1/4&\|z-w\|=1,\\
    0&\text{otherwise}.
\end{cases}\]
If $z\neq w$, then $z$ and $w$ cannot have any common neighbor points in $\partial V_N$. In this case, the sum in \eqref{eq:DiffConv} is restricted to $0$. If $z=w$, then we have $D_N(z,z)=r(z)/16\le1$, where $r(z)$ is defined in \eqref{eq:posdef}. Hence, we obtain
\begin{equation}\label{eq:boundarydecompose}
    \begin{aligned}
        \bar{G}_N D_N \bar{G}_N(x,x)&=\sum_{z,w}\bar{G}_N(x,z)D_N(z,w)\bar{G}_N(w,x)\\
        &\le\sum_{z\in\partial^- V_N}\bar{G}_N(x,z)^2\le\sum_{z\in\partial^- V_N}G^0_N(x,z)^2,
    \end{aligned}
\end{equation}
where the last inequality follows from
\[\bar{G}_N=G^0_N-G^m_N,\quad G^m_N(x,y)\ge0\text{ for all }x,y\in V_N.\]

The two-dimensional massless Green function satisfies the following classical upper bound
\begin{equation}\label{eq:lapgreenbound}
    |(-\Delta_N)^{-1}(z,w)|\lesssim\min\left[\frac{d(z)d(w)}{\|z-w\|^2+1},\log\left(\frac{\operatorname{diam}(V_N)}{\|z-w\|+1}\right)\right],
\end{equation}
which follows from standard discrete potential-theoretic arguments based on the discrete Harnack principle. Here we abbreviate $d(z)\coloneq \operatorname{dist}(z,V_N^c)$. Using this, \eqref{eq:boundarydecompose} is estimated by 
\begin{equation}\label{eq:boundarybound}
    \begin{aligned}
        \sum_{z\in \partial^- V_N}G^0_N(x,z)^2&\lesssim \sum_{z\in \partial^- V_N} \frac{1}{N^2}=O(N^{-1}).
    \end{aligned}
\end{equation}
Consequently, by \eqref{eq:varerr}, \eqref{eq:bound} and \eqref{eq:boundarybound}, we have
\begin{equation}\label{eq:varbound}
    |\bar{G}_N(x,x)-G_N(x,x)|\lesssim N^{\lambda-1}.
\end{equation}
\eqref{eq:covineq2} and \eqref{eq:varbound} yield the desired result in \eqref{eq:AuxApprox1}, and this completes the proof of \Cref{thm:AuxApprox}.

\subsection{MM Dominant Regime}\label{ch:MMdom}
In \Cref{ch:DGFFdom}, we use the auxiliary approximation by $\bar{G}_N$, since the analysis for the bi-Laplacian is technically more involved. In this section, in contrast, we directly compare $G_N$ and $\Gamma_N\coloneq N^{-\lambda}G^\Delta_N$, where $G^\Delta_N$ is the Green function of the discrete bi-Laplacian, defined in \eqref{eq:membrane}.

The matrix $L_N-N^\lambda \Delta^2_N=-\Delta_N$ is positive semidefinite, and hence $\Gamma_N-G_N$ is positive semidefinite, using the Löewner order argument again. Analogues of \eqref{eq:covineq1} and \eqref{eq:covineq2} show that it suffices to estimate the diagonal entries in the bulk.

From the identity
\begin{equation}\label{eq:beforeNeumann}
    G_N=(I-(N^\lambda \Delta_N^2)^{-1}\Delta_N)^{-1}(N^\lambda \Delta_N^2)^{-1}=(I-\Gamma_N\Delta_N)^{-1}\Gamma_N,
\end{equation}
we aim to prove its Neumann series representation. More precisely, we divide the goal of this section into two parts, the Neumann series representation and its estimation.

We have the following proposition.
\begin{proposition}\label{thm:Neumann}
    In \eqref{eq:beforeNeumann}, the $L^2$ matrix norm of $\Gamma_N\Delta_N$ is less than $1$, if $N$ is large enough. In particular, the series
    \begin{equation}\label{eq:Neumann}
        G_N=\sum_{n\ge0} (\Gamma_N \Delta_N)^n \Gamma_N=\sum_{n\ge0} \Gamma_N (\Delta_N \Gamma_N)^n
    \end{equation}
    converges in operator norm.
\end{proposition}
\begin{proof}
    For any $f:\mathbb{Z}^2\to\mathbb{R}$ zero outside $V_N$, let $g=\Gamma_N\Delta_N f$, and $g$ is zero outside $V_N$. Taking inner product of $g$ and $\Delta_N f$, we have
    \begin{equation}\label{eq:NeumannUpper}
        \begin{aligned}
            \langle g,\Delta_N f\rangle&=\langle \Delta_N g,f\rangle\\
            &\le \|\Delta_N g\|_{L^2(V_N)}\|f\|_{L^2(V_N)}.
        \end{aligned}
    \end{equation}
    On the other hand, by \eqref{eq:posdef},
    \begin{equation}\label{eq:NeumannLower}
        \begin{aligned}
            \langle g,\Delta_N f\rangle&=\langle g,N^\lambda \Delta^2_N g\rangle\\
            &\overset{\mathclap{\eqref{eq:posdef}}}{=}N^\lambda \langle \Delta_N g,\Delta_N g\rangle +\frac{1}{16}N^\lambda \sum_{z\in V_N} r(z)g(z)^2\\
            &\ge N^\lambda \|\Delta_N g\|_{L^2(V_N)}^2.
        \end{aligned}
    \end{equation}
    Combining \eqref{eq:NeumannUpper} and \eqref{eq:NeumannLower}, an inequality $\|\Delta_N g\|_{L^2(V_N)}\le N^{-\lambda} \|f\|_{L^2(V_N)}$ is derived, and
    \begin{equation}\label{eq:normineq}
        \|g\|_{L^2(V_N)}\le\|(-\Delta_N)^{-1}\|_{L^2(V_N)\to L^2(V_N)}\|(-\Delta_N)g\|_{L^2(V_N)}\lesssim N^{2-\lambda}\|f\|_{L^2(V_N)}.
    \end{equation}
    Here $\|(-\Delta_N)^{-1}\|_{L^2(V_N)\to L^2(V_N)}\asymp N^2$ is from the largest eigenvalue scaling of $(-\Delta_N)^{-1}$; see \eqref{eq:eigen}.
    
    From the definition of matrix norm, by \eqref{eq:normineq} we have
    \[\|\Gamma_N \Delta_N\|_{L^2(V_N)\to L^2(V_N)}=\sup_{f\neq0} \frac{\|\Gamma_N \Delta_N f\|_{L^2(V_N)}}{\|f\|_{L^2(V_N)}}=\sup_{f\neq0}\frac{\|g\|_{L^2(V_N)}}{\|f\|_{L^2(V_N)}}\lesssim N^{2-\lambda}.\]
    Since $\lambda>2$, we conclude that $\|\Gamma_N \Delta_N\|_{L^2(V_N)\to L^2(V_N)}\ll1$ if $N$ is sufficiently large. Hence, the series \eqref{eq:Neumann} converges in the operator norm topology.
\end{proof}
\begin{remark}
    The purpose of the operator norm estimate is not to obtain a sharp bound directly, but to justify the Neumann-series representation of $G_N$, which will then be estimated componentwise.
\end{remark}

Since we treat finite-dimensional matrices, \eqref{eq:Neumann} also implies componentwise convergence. Therefore we have
\begin{equation}
    G_N(z,w)-\Gamma_N(z,w)=\sum_{n\ge1} [ \Gamma_N (\Delta_N\Gamma_N)^n](z,w)
\end{equation}
for any $z,w\in V_N$. 

The next step is to estimate $[\Gamma_N (\Delta_N \Gamma_N)^n](x,x)$ for each $n\ge1$. We abbreviate $R_0\coloneq \Gamma_N$ and $R_{n+1}\coloneq R_n \Delta_N \Gamma_N$, for $n\ge0$. For $n\ge1$, we have
\begin{equation}\label{eq:NeumannCW}
    \begin{aligned}
        |R_n(x,x)|&=\left|\sum_{y_1,\dots,y_n\in V_N}\Gamma_N(x,y_1)(\Delta_N\Gamma_N)(y_1,y_2)\dots(\Delta_N\Gamma_N)(y_n,x)\right|\\
        &\le\sum_{y_1,\dots,y_n \in V_N} |\Gamma_N(x,y_1)||(\Delta_N\Gamma_N)(y_1,y_2)|\dots|(\Delta_N\Gamma_N)(y_n,x)|.
    \end{aligned}
\end{equation}
In \autocite[Theorem 1]{muller2019estimates}, there are the following inequalities
\begin{equation}\label{eq:MMbound}
    \begin{aligned}
        |G^\Delta_N(z,w)|&\lesssim \min \left(d(z)d(w),\frac{d(z)^2 d(w)^2}{\|z-w\|^2+1}\right),\\
        |\Delta G^\Delta_N(z,w)|&\lesssim \log\left(1+\frac{d(w)^2}{\|z-w\|^2+1}\right),
    \end{aligned}
\end{equation}
where $d(z)=\operatorname{dist}(z,V_N^c)$, and the Laplacian is taken over the variable $z$. Using these inequalities, we have $|\Gamma_N(\cdot,\cdot)|\lesssim N^{2-\lambda}$ and $|\Delta_N \Gamma_N(\cdot,\cdot)|\lesssim N^{-\lambda}\log N$ on $V_N$. With \eqref{eq:NeumannCW}, we conclude that 
\begin{equation}\label{eq:ptwisebd}
    \begin{aligned}
        |R_n(x,x)|&\lesssim N^{2n}\times N^{2-\lambda}\times (N^{-\lambda}\log N)^n\\
        &\le CN^{2-\lambda}(cN^{2-\lambda}\log N)^n 
    \end{aligned}
\end{equation}
for $n\ge1$ and some constants $c,C>0$. Therefore, we have
\begin{equation}
    \begin{aligned}
        |G_N(x,x)-\Gamma_N(x,x)|&\le\sum_{n\ge1}|R_n(x,x)|\lesssim N^{2-\lambda}\sum_{n\ge1} (cN^{2-\lambda}\log N)^n\\
        &=N^{2-\lambda}(1+O(N^{2-\lambda}\log N))O(N^{2-\lambda}\log N)\\
        &=O(N^{4-2\lambda}\log N),
    \end{aligned}
\end{equation}
and the error $O(N^{4-2\lambda}\log N)$ is the desired estimate in \Cref{thm:domregime}.

    \appendix
    % !TEX root = ../main.tex
\section{Massive Green Function Analysis}\label{section:massiveanalysis}
The main purpose of this section is to construct an appropriate estimate for $G^m$. By translation invariance of the symmetric random walk, we can write $G^m(x,y)=G^m(x-y,0)$, and thus we may assume $y=0$. Further we abbreviate $G^m(x)=G^m(x,0)$ throughout this section. This section is composed of two parts, the off-diagonal estimate (i.e., $x\neq0$) and the on-diagonal estimate (i.e., $x=0$).

\subsection{Off-diagonal Estimate}
In this section, we prove \Cref{thm:massiveanalysis}. Since \Cref{thm:compregime} distinguishes between the regimes $\|x-y\| \gtrsim N^{\lambda/2}$ and $\|x-y\|\ll N^{\lambda/2}$, we derive estimates adapted to each regime.

The long-range asymptotic behavior of the massive Green function was studied in \autocite{michta2022asymptotic}, where precise direction-dependent asymptotics were obtained in several scaling regimes. However, those asymptotic results are not directly applicable to our setting, since in our model we only assume $\|x\|\asymp N$, while the direction of $x$ may vary with $N$. For our purposes, precise asymptotics are unnecessary; uniform upper bounds are sufficient. The goal of this section is to establish the uniform estimates \Cref{thm:massiveanalysis}.

We distinguish between statements with and without a difference structure. If it does not include such structure, then its proof becomes simple by the universal Gaussian estimate 
\begin{equation}\label{eq:prob_bdd}
    \mathbb{P}(S_n=x)\le C \frac{e^{-c\|x\|^2/n}}{n^{d/2}}.
\end{equation}
It is a classical estimate; see, e.g., \autocite{hebisch1993gaussian}. 

Since $m$ is bounded, we may assume $m\in(0,M)$ for some fixed $M>0$. For such $M$, there is $c=c_M\in(0,1)$ satisfying $\log (1+m^2)\ge cm^2$ for all $m\in(0,M)$. We can therefore derive a bound for $\rho^n=(1+m^2)^{-n}$:
\begin{equation}\label{eq:survrate_bdd}
    \rho^n=(1+m^2)^{-n}=e^{-n\log (1+m^2)}\le e^{-cm^2 n}.
\end{equation}
Combining these two estimates, for $x\neq0$, the following estimate
\begin{equation}\label{eq:massivebound}
    G^m(x)=\sum_{n\ge1}\rho^{n+1} \mathbb{P}(S_n=x)\le\sum_{n\ge1} C n^{-d/2} e^{-c(\|x\|^2/n +m^2n)}
\end{equation}
is obtained. In \eqref{eq:massivebound}, we do not consider $n=0$ in the sum, since $\mathbb{P}(S_0=x)=0$ unless $x=0$.

In contrast, if the statement includes a difference structure, we cannot directly use the heat-kernel estimates such as \eqref{eq:prob_bdd}, since the parity of $\|x\|_1$ might cause some issues. Instead of using this, we use the Fourier transform for $\mathbb{P}(S_n=x)$, 
\begin{equation}\label{eq:prob_fourier}
    \mathbb{P}(S_n=x)=\frac{1}{(2\pi)^d}\int_{[-\pi,\pi]^d} \psi(k)^n e^{-ik\cdot x}\mathrm{d}k,
\end{equation}
where, for $k=(k_1,\dots,k_d)$,
\[\psi(k)\coloneq \frac{1}{d}\sum_{j=1}^d \cos(k_j)\]
is the characteristic function of $X_1$ in \Cref{ch:RWintro}. Hence, $\psi^n$ is the characteristic function of $S_n$.

For the rest of this section, we write $\mathbb{P}(S_n=x)$ as $p_n(x)$. We fix $d=2$ in the formulas above such as \eqref{eq:prob_bdd} or \eqref{eq:massivebound}.

\subsubsection{Subcritical Massive Analysis}\label{ch:submass}
We use \eqref{eq:massivebound}. We assume that $m$ is bounded and $m\|x\|\to\infty$. For the function $f(t)\coloneq \|x\|^2/t+m^2 t$ on $t>0$, the function $e^{-f(t)}$ is integrable on $[1,\infty)$. We therefore have an inequality
\begin{equation}\label{eq:subcritint}
    \eqref{eq:massivebound}\lesssim \int_1^\infty t^{-1}e^{-c (\|x\|^2/t +m^2t)}\mathrm{d}t=\int_1^\infty t^{-1}e^{-cm\|x\| (\|x\|/mt +mt/\|x\|)}\mathrm{d}t.
\end{equation}
We substitute $s=mt/\|x\|$ in the integral of \eqref{eq:subcritint}, and we have to estimate 
\begin{equation}\label{eq:subcritint2}
    \int_{m/\|x\|}^\infty s^{-1}e^{-cm\|x\| (1/s +s)}\mathrm{d}s.
\end{equation}
Since $m$ is bounded and $m\|x\|\to\infty$, we also have $m/\|x\|\to0$. Indeed, the integrand in \eqref{eq:subcritint2} is integrable on $(0,\infty)$. The integral can be estimated using the integral representation of Bessel function (see, e.g., \autocite[(10.32.10)]{NIST:DLMF})
\begin{equation}\label{eq:BesselRep}
    K_\nu(z)=\frac{1}{2}\left(\frac{z}{2}\right)^\nu\int_0^\infty e^{-t-z^2/4t}t^{-\nu-1}\mathrm{d}t,
\end{equation}
where $K_\nu(z)$ is the modified Bessel function of the second kind. The function $K_\nu(z)$ has an asymptotic (see, e.g., \autocite[(10.25.3)]{NIST:DLMF})
\begin{equation}\label{eq:BesselAsymp}
    K_\nu(z)\simeq \sqrt{\frac{\pi}{2z}}e^{-z},\quad z\to\infty,\text{ $\nu$ fixed.}
\end{equation}

From these properties of the Bessel function, we deduce that
\begin{equation}\label{eq:subcritint3}
    \eqref{eq:massivebound}\lesssim \int_0^\infty s^{-1}e^{-cm\|x\| (1/s +s)}\mathrm{d}s=2K_0(cm\|x\|)\asymp (m\|x\|)^{-1/2}e^{-cm\|x\|}.
\end{equation}

\subsubsection{Critical Massive Analysis}\label{ch:critmass}
Here we assume that $m\|x\|\to s$ for some $s>0$ with $\|x\|\to\infty$. When estimating $G^m(x)$, we use \eqref{eq:massivebound}.

\paragraph{Pointwise Estimate}
We divide the range of the series \eqref{eq:massivebound} by $n\le \|x\|^2$ and $n>\|x\|^2$. On $n\le\|x\|^2$, we use the trivial inequality $e^{-cm^2n}\le 1$. We have
\begin{equation}\label{eq:critmassbound_small}
    \sum_{n\le\|x\|^2}n^{-1} e^{-C\|x\|^2/n -cm^2n}\le \sum_{n\le\|x\|^2}n^{-1} e^{-C\|x\|^2/n}.
\end{equation}
Since $e^{-t}\lesssim t^{-1}$ uniformly in $t>0$, \eqref{eq:critmassbound_small} is bounded by
\begin{equation}\label{eq:critmassbound_small2}
    \sum_{n\le\|x\|^2}n^{-1} e^{-C\|x\|^2n}\lesssim \sum_{n\le\|x\|^2}\|x\|^{-2}= O(1).
\end{equation}

On $n>\|x\|^2$, we bound $e^{-C\|x\|^2/n}$ by $1$. Since $m=s\|x\|^{-1}(1+o(1))$,
\begin{equation}\label{eq:critmassbound_large}
    \sum_{n>\|x\|^2}n^{-1} e^{-C\|x\|^2/n -cm^2n}\le \sum_{n>\|x\|^2}n^{-1} e^{-c'\|x\|^{-2}n}
\end{equation}
for some $c'>0$. Using again $e^{-t}\lesssim t^{-1}$ uniformly in $t>0$, 
\begin{equation}\label{eq:critmassbound_large2}
    \sum_{n>\|x\|^2}n^{-d/2} e^{-c'\|x\|^{-2}n}\lesssim \sum_{n>\|x\|^2}\|x\|^2 n^{-2}\asymp \|x\|^2\cdot \|x\|^{-2}=O(1).
\end{equation}
Hence \eqref{eq:critmassbound_small2} and \eqref{eq:critmassbound_large2} imply the desired result. This proves the case $\alpha=0$.

\paragraph{Gradient Estimate}\label{ch:critmass_grad}
When estimating $\nabla^\alpha G^m(x)$ for a multiindex $\alpha$ with $|\alpha|>0$, we use Fourier transform \eqref{eq:prob_fourier}. We have
\[\nabla^\alpha p_n(x)=\frac{1}{(2\pi)^2}\int_{[-\pi,\pi]^2 }\prod_j(e^{-i k\cdot e_j }-1)^{\alpha_j}\psi(k)^n e^{-ik\cdot x}\mathrm{d}k\]
and hence,
\begin{equation}\label{eq:gradfourier}
    \begin{aligned}
        \nabla^\alpha G^m(x)&=\frac{1}{(2\pi)^2}\sum_n \rho^{n+1} \int_{[-\pi,\pi]^2}\prod_j(e^{-i k\cdot e_j }-1)^{\alpha_j}\psi(k)^ne^{-ik\cdot x}\mathrm{d}k\\
        &=\frac{\rho}{(2\pi)^d} \int_{[-\pi,\pi]^2}\left(\sum_n \rho^n \psi(k)^n\right) \prod_j(e^{-i k\cdot e_j }-1)^{\alpha_j}e^{-ik\cdot x}\mathrm{d}k.
    \end{aligned}
\end{equation}
In \eqref{eq:gradfourier}, the summation and integration can be exchanged since
\begin{equation}\label{eq:ratio}
    \left|\rho^n \psi(k)^n e^{-ik\cdot x}\prod_j(e^{-ik\cdot e_j}-1)^{\alpha_j}\right|\le 2^{|\alpha|}\rho^n,
\end{equation}
which is summable over $n$ unless $m=0$.

Since $|\rho|<1$, applying geometric series identity on \eqref{eq:gradfourier} yields
\begin{equation}\label{eq:gradfourier2}
    \text{\eqref{eq:gradfourier}}=\frac{\rho}{(2\pi)^2} \int_{[-\pi,\pi]^2}\frac{e^{-ik\cdot x}\prod_j(e^{-i k\cdot e_j }-1)^{\alpha_j}}{1-\rho\psi(k)}\mathrm{d}k.
\end{equation}
In this representation, as the integrand is $2\pi$ periodic in each coordinate, all boundary terms cancel in the integration by parts.

We apply integration by parts to \eqref{eq:gradfourier2}. Without loss of generality, assume that $|x_1|\ge|x_2|$. Iterating the identity $\partial_{k_1}e^{-ik\cdot x}=-ix_1e^{-ik\cdot x}$ $|\alpha|+1$ times yields
\begin{equation}\label{eq:gradfourier3}
    \begin{aligned}
        \text{\eqref{eq:gradfourier2}}&=\frac{\rho}{(2\pi)^2}\frac{1}{(-ix_1)^{|\alpha|+1}} \int_{[-\pi,\pi]^2}\bigl(\partial_{k_1}^{|\alpha|+1}e^{-ik\cdot x} \bigr)\frac{\prod_j(e^{-i k\cdot e_j }-1)^{\alpha_j}}{1-\rho\psi(k)}\mathrm{d}k\\
        &=\frac{\rho}{(2\pi)^2}\frac{1}{(ix_1)^{|\alpha|+1}} \int_{[-\pi,\pi]^2}e^{-ik\cdot x} \left(\partial_{k_1}^{|\alpha|+1}\frac{\prod_j(e^{-i k\cdot e_j }-1)^{\alpha_j}}{1-\rho\psi(k)}\right)\mathrm{d}k.
    \end{aligned}
\end{equation}
Now we abbreviate $p_\alpha(k)\coloneq \prod_j(e^{-i k\cdot e_j }-1)^{\alpha_j}$ and $q_m(k)\coloneq 1-\rho\psi(k)$. (Recall that $\rho=(1+m^2)^{-1}$.) From \eqref{eq:gradfourier3}, we have
\[|\nabla^\alpha G^m(x)|\lesssim \frac{\rho}{|x_1|^{|\alpha|+1}}\left\|\partial_{k_1}^{|\alpha|+1}\frac{p_\alpha}{q_m}\right\|_{L^1([-\pi,\pi]^2)}.\]

Since $\psi(k)=1-\|k\|^2/4+O(\|k\|^4)$ for $\|k\|\to 0$, for a small $\delta>0$, there exists some constants $c,C>0$ depending on $\delta$, such that
\begin{equation}\label{eq:denomasymp}
    1-\rho+\rho c\|k\|^2\le 1-\rho\psi(k)\le 1-\rho + \rho C\|k\|^2
\end{equation}
for $\|k\|\le \delta$. For such $\delta$, we split the integral in \eqref{eq:gradfourier2} into $E_1\coloneq \{k:\|k\|\le\delta\}$ and $E_2$ being its complement.

We have the following estimates near $k=0$:
\begin{equation}\label{eq:IBPestimate}
    |\partial_{k_1}^r p_\alpha(k)|\lesssim \|k\|^{(|\alpha|-r)_+},\quad |\partial_{k_1}^r q_m(k)^{-1}|\lesssim (m+\|k\|)^{-2-r},
\end{equation}
where $(a)_+\coloneq \max(a,0)$. The first estimate follows directly from the fact that $p_\alpha$ has a zero of order $|\alpha|$ at the origin. The second follows by induction from the following estimates near $k=0$:
\[q_m(k)\asymp m^2+\|k\|^2,\quad |\partial^r q_m(k)|\lesssim \begin{cases}
    \|k\|&r=1,\\
    1&r\ge2.
\end{cases}\]
By the Leibniz rule,
\begin{align*}
    |\partial^{|\alpha|+1}(p_\alpha q_m^{-1})|&=\left|\sum_{r=0}^{|\alpha|+1}\binom{|\alpha|+1}{r}(\partial^r p_\alpha)(\partial^{|\alpha|+1-r}q_m^{-1})\right|\\
    &\lesssim \sum_{r=0}^{|\alpha|+1}\binom{|\alpha|+1}{r}\|k\|^{(|\alpha|-r)_+}(m+\|k\|)^{-|\alpha|+r-3}\\
    &\lesssim \sum_{r=0}^{|\alpha|}(m+\|k\|)^{-3}+(m+\|k\|)^{-2}\\
    &\lesssim (m+\|k\|)^{-3}+(m+\|k\|)^{-2}.
\end{align*}
Hence the integral on $E_1$ is bounded by
\begin{equation}\label{eq:fouriersmall}
    \begin{aligned}
        \int_{E_1}\left|\partial_{k_1}^{|\alpha|+1}\frac{p_\alpha(k)}{q_m(k)}\right|\mathrm{d}k&\lesssim \int_{E_1}(m+\|k\|)^{-3}+(m+\|k\|)^{-2}\mathrm{d}k\\
        &=\int_0^{2\pi}\int_0^\delta \left(\frac{1}{(m+r)^3}+\frac{1}{(m+r)^2}\right)r\mathrm{d}r\mathrm{d}\theta\\
        &\lesssim m^{-1}+|\log m|\lesssim m^{-1}.
    \end{aligned}
\end{equation}
On $E_2$, since $q_m(k)>c>0$ uniformly for some $c$, $\partial^{|\alpha|+1}(p_\alpha/q_m)$ is bounded in absolute value by a positive constant $C_{\alpha,\delta}$. As $E_2$ is a bounded region, the integral on $E_2$ is $O(1)\lesssim m^{-1}$. 

In conclusion, we have
\[|\nabla^\alpha G^m(x)|\lesssim \rho|x_1|^{-|\alpha|-1}m^{-1}\asymp \frac{\rho}{s}|x_1|^{-|\alpha|}\lesssim \|x\|^{-|\alpha|},\]
where we used $\|x\|\le \sqrt{2}|x_1|$.

\subsubsection{Supercritical Massive Analysis}\label{ch:supmass}
We analyze the behavior of the regime $m\|x\|\to0$. As already mentioned in \Cref{ch:randomwalk}, $G^0_\Lambda$ does not converge when we take $\Lambda\uparrow\mathbb{Z}^d$ limit. To avoid this issue, subtracting $G^0_\Lambda(0)$ before taking limit should be considered, and its limit is the potential kernel $\mathfrak{a}$. We use the same idea for $G^m$.

Let $J(m)=J(m;x)\coloneq G^m(x)-G^m(0)$. Using the random walk representation,
\begin{equation}
    J(m;x)=\sum_{n\ge0}\rho^{n+1} (p_n(x)-p_n(0)).
\end{equation}
The main goal is to analyze
\begin{equation}\label{eq:supercritmass_RW}
    J(m;x)-\mathfrak{a}(x)=\sum_{n\ge0}(\rho^{n+1}-1)(p_n(x)-p_n(0)).
\end{equation}

We use the same argument in \Cref{ch:critmass_grad}. The geometric series for $\psi(k)^n$ is valid except when $|\psi(k)|=1$. It happens only when $k=0$ or $k\in\{\pm\pi\}^2$. For the latter case, the geometric series identity $\sum(\psi(k))^n=(1-\psi(k))^{-1}$ fails at such points. This does not affect the integral, since $(1-\psi(k))^{-1}$ is defined there and the exceptional set is finite and hence has zero measure.

In contrast, $k=0$ may cause a singularity. On a neighborhood of $k=0$, however, we have
\[\left|\frac{e^{-ik\cdot x}-1}{1-\psi(k)}\right|\lesssim \frac{\|k\|\|x\|}{\|k\|^2}=\|x\|\|k\|^{-1},\]
and $\|k\|^{-1}\in L^1_{\mathrm{loc}}(\mathbb{R}^2)$. Therefore, the integrand is locally integrable near $k=0$, and we may rewrite \eqref{eq:supercritmass_RW} as
\begin{equation}\label{eq:supercrit_fourier}
    \frac{1}{(2\pi)^d} \int_{[-\pi,\pi]^2}(e^{-ik\cdot x}-1)\left(\frac{\rho}{1-\rho\psi(k)}-\frac{1}{1-\psi(k)}\right)\mathrm{d}k.
\end{equation}

We divide the domain of integral in \eqref{eq:supercrit_fourier} into $E_1$ and $E_2$ as in \Cref{ch:critmass_grad}. On $E_1$, we again substitute $k=m\xi$, and
\begin{equation}\label{eq:supercrit_fourier_small}
    \begin{aligned}
        &\left|\int_{E_1}(e^{-ik\cdot x}-1)\left(\frac{\rho}{1-\rho\psi(k)}-\frac{1}{1-\psi(k)}\right)\mathrm{d}k\right|\\
        \le\ &m^2\int_{\|\xi\|\le \delta/m}|e^{-i\xi \cdot (mx)}-1|\left|\frac{\rho}{1-\rho\psi(m\xi)}-\frac{1}{1-\psi(m\xi)}\right|\mathrm{d}\xi.
    \end{aligned}
\end{equation}
From the asymptotics $1-\rho\psi(m\xi)\asymp m^2(1+\|\xi\|^2)$ and $1-\psi(m\xi)\asymp m^2\|\xi\|^2$, the integrand in \eqref{eq:supercrit_fourier_small} has the following estimate
\begin{equation}\label{eq:supercrit_fourier_small2}
    \begin{aligned}
        \left|\frac{\rho}{1-\rho\psi(m\xi)}-\frac{1}{1-\psi(m\xi)}\right|&=\frac{|\rho-1|}{(1-\rho\psi(m\xi))(1-\psi(m\xi))}\\
        &\asymp \frac{1}{m^2(1+\|\xi\|^2)\|\xi\|^2}. 
    \end{aligned}
\end{equation}
Combining \eqref{eq:supercrit_fourier_small2} with the bound $|e^{-i\xi\cdot (mx)}-1|\le \min(2,\|\xi\|\|mx\|)$, we deduce that
\begin{equation}
    \begin{aligned}
        \text{\eqref{eq:supercrit_fourier_small}}&\lesssim\int_{\|\xi\|\le\delta/m}\frac{\min(1,\|\xi\|\|mx\|)}{(1+\|\xi\|^2)\|\xi\|^2}\mathrm{d}\xi=\int_0^{2\pi}\int_0^{\delta/m}\frac{\min(1,r\|mx\|)}{(1+r^2)r^2}r\mathrm{d}r\mathrm{d}\theta\\
        &\le\|mx\|\int_0^{2\pi}\int_0^{\delta/m}\frac{1}{1+r^2}\mathrm{d}r\mathrm{d}\theta=O(\|mx\|).
    \end{aligned}
\end{equation}

On $E_2$, both $1-\rho\psi(k)$ and $1-\psi(k)$ are bounded away from $0$ uniformly. Hence, we obtain
\begin{equation}
    \left|\frac{\rho}{1-\rho\psi(k)}-\frac{1}{1-\psi(k)}\right|=\frac{|\rho-1|}{|1-\rho\psi(k)||1-\psi(k)|}\lesssim |\rho-1|\asymp m^2,
\end{equation}
and
\begin{equation}
    \left|\int_{E_2}(e^{-ik\cdot x}-1)\left(\frac{\rho}{1-\rho\psi(k)}-\frac{1}{1-\psi(k)}\right)\mathrm{d}k\right|\lesssim m^2 \int_{E_2}|e^{-ik\cdot x}-1|\mathrm{d}k\lesssim m^2.
\end{equation}
Therefore we conclude that $|J(m;x)-\mathfrak{a}(x)|=o(1)$. 

This completes the proof of \Cref{thm:massiveanalysis}.

\subsection{On-diagonal Estimate}\label{ch:originvar}
As already mentioned in \Cref{ch:shortrange}, the estimate for $G^m(0,0)$ in $d=2$ can be improved as
\[G^m(0,0)=\frac{2}{\pi} |\log m|+O(1), \quad m\to 0+.\]
In this section, we derive this estimate.

The local central limit theorem (see, e.g., \autocite[Theorem 1.2.1]{lawler2013intersections}) implies 
\[p_{2n}(0)=\frac{1}{\pi n}+O(n^{-2}),\]
uniformly in $n>0$. It is sufficient to consider only the even indices, since the random walk cannot visit the origin at odd times.

Substituting this into \eqref{eq:massivewholespace}, we obtain
\[G^m(0,0)=\rho\sum_{n\ge0}\rho^{2n} p_{2n}(0)=\rho+ \rho\sum_{n\ge1}\rho^{2n}\left(\frac{1}{\pi n}+O(n^{-2})\right).\]
In the summand, $\rho^{2n} n^{-2}$ is summable and the sum is $O(1)$, which is the desired rate for the error. Hence we absorb the $O(n^{-2})$ contribution into the $O(1)$ error. Therefore,
\[\rho\sum_{n\ge1}\frac{\rho^{2n}}{\pi n}=-\frac{\rho}{\pi}\log (1-\rho^2).\]
Put $\rho=1-m^2+O(m^4)$ and $1-\rho^2=2m^2(1+o(m^2))$. We obtain
\[\rho \log (1-\rho^2) =(1-m^2 +O(m^4))\log (m^2 (2+o(1)))=2\log m +O(1),\]
since $m^2\log m\to 0$, as $m\to0+$. Hence
\[G^m(0,0)=-\frac{2}{\pi}\log m+O(1)=\frac{2}{\pi}|\log m|+O(1).\]

    \section*{Statements and Declarations}

    \subsection*{Data Availability}
    Data sharing is not applicable to this article as no datasets were generated or analyzed during the current study.

    \subsection*{Competing Interests}
    The author declares no competing interests.

    \subsection*{Acknowledgments}
    The author thanks the anonymous referees for their careful reading of the manuscript and their detailed and constructive comments.

    \printbibliography
\end{document}